%% file: dirichlet_boundary_control_LDG.tex
\documentclass[a4paper,10pt]{article}

\usepackage{todonotes}
\newcommand{\mh}[1]{\todo[inline,color=teal!30]{MH: #1}}
\newcommand{\h}[2][cyan]{\emph{\textcolor{#1}{#2}}}
\newcommand{\hy}[1]{\todo[inline,color=gray!30]{HY: #1}}

\usepackage{a4wide}
\usepackage[colorlinks]{hyperref}
\usepackage{amsmath}
\usepackage{amsthm}
\usepackage{amsfonts}
\usepackage{amssymb}
\usepackage{mathtools}
\usepackage{algorithm}
\usepackage{algorithmic}
\usepackage{epsfig}
\usepackage[utf8]{inputenc}
\usepackage{graphics}
\usepackage{caption}
\usepackage{subcaption}
\usepackage{xcolor}
\usepackage{booktabs}
\usepackage{tcolorbox}
\usepackage{listings}
\usepackage{color} 
\usepackage{enumitem}
\usepackage{tikz}
\usetikzlibrary{calc,positioning,fit,backgrounds}
\usetikzlibrary{3d}
\definecolor{mygreen}{RGB}{28,172,0} 
\definecolor{mylilas}{RGB}{170,55,241}
\usepackage{fancyhdr} 
\newcommand{\set}[2]{\left\{{#1}\,:~{#2}\right\}}
\newcommand {\average}[1] {\mbox{$\left\{\!\!\left\{ #1 \right\}\!\!\right\}$}}
\newcommand {\jump}[1] {\mbox{$\left[\!\left[ #1 \right]\!\right]$}}

\DeclareMathOperator{\sign}{sign}

\newcommand{\Uad}{{U^{\textup{ad}}}}
\newcommand{\Uadh}{{U^{\textup{ad}}_h}}
\newcommand{\hF}{{\frac{1}{2}}}
\newcommand{\cT}{{\mathcal{T}}}
\newcommand{\cE}{{\mathcal{E}}}
\newcommand{\bq}{{\mathbf{q}}}
\newcommand{\br}{{\mathbf{r}}}
\newcommand{\bw}{{\mathbf{w}}}
\newcommand{\bp}{{\mathbf{p}}}
\newcommand{\bn}{{\mathbf{n}}}
\newcommand{\bW}{{\mathbf{W}}}

\newtheorem{theorem}{Theorem}[section] 
\newtheorem{lemma}[theorem]{Lemma}
\newtheorem{assumption}[theorem]{Assumption}

\newtheorem{remark}[theorem]{Remark}
\numberwithin{theorem}{section}

\newcommand{\x}[2][blue]{#2}
\newcommand{\z}[2][blue]{#2}
\newcommand{\zz}[2][red]{#2}
\author{Peter Benner\footnote{Max Planck Institute for Dynamics of Complex Technical Systems, Magdeburg, Germany, \url{benner@mpi-magdeburg.mpg.de}} \and Michael Hinze\footnote{Mathematisches Institut, Universität Koblenz-Landau, Koblenz, Germany, \url{hinze@uni-koblenz.de}} \and Hamdullah Yücel \footnote{Institute of Applied Mathematics, Middle East Technical University, Ankara, Türkiye, \url{yucelh@metu.edu.tr}}\\[1cm] Dedicated to our long-standing, highly esteemed colleague and friend Ronald Hoppe}

\title{A Local Discontinuous Galerkin Method for Dirichlet Boundary Control Problems}

\begin{document}

\maketitle

\begin{abstract}
In this paper, we consider control constrained $L^2-$Dirichlet boundary control of a convection-diffusion equation on a two dimensional convex polygonal domain. We discretize the control problem based on the local discontinuous Galerkin method with piecewise linear ansatz functions for the flux and potential. We derive a priori error estimates for the full as well as for the variational discrete control approximation. We present a selection of numerical results to demonstrate the performance of our approach and to underpin the theoretical findings.
\end{abstract}


\input{R2-introduction}

\input{R2-problem}

\input{R2-DG}

\input{R2-estimate_alternative}  
\input{R2-numeric}

\input{R2-conclusion}


\end{document}

%% file: R2-introduction.tex
\section{Introduction}\label{sec:introduction}

Let $\Omega \subset \mathbb{R}^2 $ be a bounded, convex open domain with a Lipschitz boundary $\Gamma$. In the present work, we are interested in the numerical analysis of the Dirichlet boundary control problem
\begin{equation}\label{p1}
       \underset{u \in \Uad}{\hbox{ minimize }} \;  \frac{1}{2}\|y-y^d\|^{2}_{0,\Omega} + \frac{\omega}{2} \|u\|^{2}_{0,\Gamma} =: J(y,u)
\end{equation}
subject to
\begin{equation}\label{p2}
 \begin{split}
  \nabla \cdot ( - \epsilon \nabla y + \beta  y )  + \alpha y =& \, f      \qquad \qquad \hbox{in   }  \Omega,   \\
     y=& \, u    \qquad \qquad  \hbox{on   }   \Gamma, 
\end{split}
\end{equation}
where the admissible control set $\Uad$  is specified by
\begin{align}\label{p3}
\Uad := \{ u \in L^{2}(\Gamma): \;  u_{a} \leq u(x) \leq u_{b} \;\; \hbox{a.e. on} \;\; \Gamma\},
\end{align}
with the real numbers $u_a$ and $u_b$ satisfying $u_a \leq u_b$. For the  discretization of this control problem, we consider  variational discretization \cite{MHinze_2005a}, as well as  a  discontinuous Galerkin approximation of the control variable \x{with piecewise linear finite elements}. In both cases, we employ the local discontinuous Galerkin (LDG) method for the discretization of the state. \x{For the notation and conventions, we refer to Sections~\ref{sec:regularity} and~\ref{sec:ldg}}.

\paragraph*{Contribution:} For both control schemes, we present a general estimate for the control and state errors. For the variational discrete scheme, in Theorem \ref{thm:nodisc_adjoint_representation} for $u\in H^{s}(\Gamma)$ $(s\in [0,1])$ we obtain the error relation
\[\| u - u_h \|_{0,\Gamma}  + \|y-y_h\|_{0,\Omega} \lesssim \|y-y_h(u)\|_{0,\Omega} + \|(\bp_h(u) - \bp)\cdot\bn\|_{0,\Gamma} + \|z_h(u) - z\|_{0,\Gamma} =:E(y,\bp,z),\]
    which only involves Galerkin approximation errors for the primal potential $y$, as well as for the adjoint flux $\bp$ and the adjoint potential $z$.
For the fully discrete case, in Theorem \ref{thm:fully-discrete} we obtain the error relation
\begin{multline*}
    \| u - u_h \|_{0,\Gamma}  + \|y-y_h\|_{0,\Omega} \lesssim E(y,\bp,z) + \\ +\|\pi_hu-u\|_{0,\Gamma} +\big(\|(\bp_h - \bp(y_h))\cdot\bn\|_{0,\Gamma}+\|z_h - z(y_h)\|_{0,\Gamma}\big)^{1/2}\|\pi_hu-u\|_{0,\Gamma}^{1/2} + \|\pi_h u - u\|_{-s,\Gamma}^{1/2},
    \end{multline*}
    which, in addition to the Galerkin approximation errors for $y,z$, and $\bp$, also contains approximation errors for the auxiliary adjoint flux $\bp(y_h)$ and the potential $z(y_h)$, as well as projection errors associated with the Carstensen quasi-interpolation operator $\pi_h$. \x{It is noted that here and throughout the paper, the notation $A \lesssim B$ means that there exists a positive constant $C$, independent of the relevant discretization parameters, such that $A \leq C B$.} We would like to emphasize that our numerical analysis is inspired by the techniques introduced in \cite[Chapter 3]{MHinze_RPinnau_MUlbrich_SUlbrich_2009a} and does not rely on the stability and continuity properties of the discrete solution operator associated with the LDG approximation scheme, thus differing from common practices in the literature; compare, e.g., \cite{WGong_NYan_2011b,JPfefferer_BVexler_2025}. In the fully discrete case, for the error estimation we only rely on the uniform boundedness of the discrete optimal states $y_h$ in the $L^2(\Omega)$ norm, which can be easily deduced from the structure of the cost functional together with the boundedness of the continuous solution operator. For our LDG approach with \x{linear finite} elements to approximate the  discrete fluxes and potentials, we obtain
    \[
\| u - u_h \|_{0,\Gamma}  + \|y-y_h\|_{0,\Omega} \lesssim h^{1/2}
    \]
for both the fully discrete and the variational discrete approaches. This is in accordance with the results reported in \cite{WGong_WHu_MMateos_JSingler_XZhang_YZhang_2018} for \x{the case  where} piecewise constant fluxes, piecewise linear potentials, and piecewise linear controls for the approximation of the control problem are utilized and convergence with $h^{1/2}$ is shown, regardless of the smoothness of the involved states and adjoint variables. One conclusion from the analysis here can be that in the context of DG methods, stability can only be achieved at the expense of accuracy.

Finite element approximation of optimal control problems has been an active research area in engineering design for a long time and has been extensively studied in many scientific and engineering applications. We refer to the monographs   \cite{MHinze_RPinnau_MUlbrich_SUlbrich_2009a,FTroeltzsch_2010a,BVexler_DMeidner_2025}, as well as the references therein for the theory  of  optimal control problems and for the development of numerical methods. Compared to distributed control problems, in which the control acts throughout the domain $\Omega$, there has been relatively less research on boundary control problems. Moreover,  the majority of existing studies on  boundary control problems focus on Neumann boundary control problems (see, e.g., \cite{NArada_ECasas_FTroeltzsch_2002a,ECasas_MMateos_2008a,TGeveci_1979,MHinze_UMatthes_2009a}), where the control is applied through Neumann boundary conditions rather than through conditions of the form \eqref{p2}. In contrast to distributed and Neumann control problems, Dirichlet boundary control problems pose additional challenges in both theoretical and numerical analysis, as the Dirichlet boundary data cannot be directly incorporated into a standard variational formulation. As a result, various alternative formulations have been proposed in the literature.

The first approach involves approximating the nonhomogeneous Dirichlet boundary condition using either a Robin boundary condition or a weak boundary penalization method; see, e.g., \cite{NArada_JPRaymond_2002a,FBBelgacem_HEFekih_HMetoui_2003}. The second one, which is adopted in the present study, employs the
$L^2(\Gamma)$  norm. In this case, the traditional finite element method treats the state variable in a very weak sense, as the nonhomogeneous Dirichlet boundary condition is prescribed only in  $L^2(\Gamma)$. Extensive numerical studies have been conducted on elliptic Dirichlet boundary control problems using this formulation. We refer to \cite{TApel_MMateos_JPfefferer_ARosch_2018,ECasas_JPRaymond_2006b,KDeckelnick_AGunther_MHinze_2009,SMay_RRannacher_BVexler_2013,MMateos_2018,JPfefferer_BVexler_2025} for a priori error estimates and to \cite{TApel_MMateos_JPfefferer_ARosch_2015} for the regularity analysis. In  \cite{ECasas_JPRaymond_2006b}, a control-constrained optimal control problem governed by a semilinear elliptic equation on a convex polygonal domain was investigated, where error estimates of order $h^s$ for both the control and state variables were derived with $s < \min(1, \pi/(2\theta))$ and $\theta$ denoting the largest interior angle of the domain. Subsequently, an enhanced error estimate of order $h^{3/2}$ for the control variable in the case of smooth domains  was presented in \cite{KDeckelnick_AGunther_MHinze_2009} by exploiting the superconvergence properties of regular triangulations. Then, improved convergence rates for the state variable in the unconstrained case on convex domains were obtained in \cite{SMay_RRannacher_BVexler_2013} through a duality argument combined with a control estimate in a norm weaker than $L^2(\Gamma)$. Later, error estimates of order $h^s$, with $s < \min(1, \pi/\theta - 1/2)$ for general meshes and $s < \min(3/2, \pi/\theta - 1/2)$ for superconvergent meshes, were established for the control variable on general polygonal domains in \cite{TApel_MMateos_JPfefferer_ARosch_2018}. Another common approach to formulate Dirichlet boundary control problems is the energy space method, where the $L^2$-norm in the cost functional is replaced by the $H^{1/2}$-norm. This choice of control space allows for a standard weak formulation of the state equation and yields optimal estimates on a convex polygonal domain; however, it introduces an additional operator, namely the Steklov–Poincaré operator, through the harmonic extension of the given Dirichlet boundary data; see, e.g., \cite{GOf_TXPhan_OSteinbach_2010,MWinkler_2020}. In a similar fashion, the $H^1$-norm, which naturally results in harmonic control, is utilized in \cite{SChowdhury_TGudi_AKNandakumaran_2017,MKarkulik_2020} to regularize the control variable. Further, it is noted in \cite{SDu_XHe_2023} that a fourth-order formulation of the adjoint variable, derived by eliminating the control and state variables, can be employed to address Dirichlet boundary control problems. 

Unlike continuous finite element approximations of the state variables, alternative discretization techniques, such as the mixed finite element method \cite{WGong_NYan_2011b} and the hybridizable discontinuous Galerkin (HDG) method \cite{WWHu_JGShen_JRSingler_YWZhang_XBZheng_2020}, have been employed to overcome the challenges arising from the variational formulation of continuous finite element methods. In \cite{WGong_NYan_2011b}, error estimates of order $h^{1-1/s}$ with $s \geq 2$ were established  for polygonal domains and of order $h|\ln h|$ for smooth domains, where the classical $RT_0$ mixed finite element method was used for the approximation. In \cite{WWHu_JGShen_JRSingler_YWZhang_XBZheng_2020}, an HDG method with discrete fluxes in $P_k$ and potentials in $P_{k+1}$ for $k\ge 0$ was applied. Error estimates of order  $h^s$ for $s < \min(3/2, \pi/\theta -1/2)$ for $k\ge 1$, and of order $h^{1/2}$ in the case $k=0$, were reported. \x{Here and throughout the paper, $k$ denotes the polynomial degree of the finite element space.}

All of the aforementioned studies concentrate on Dirichlet boundary control problems for elliptic equations. However, such control problems are also of significant importance in various applications governed by fluid dynamics models, including the Stokes \cite{WGong_WHu_MMateos_JRSingler_YZhang_2020,WGong_MMateos_JSingler_YZhang_2022} and Navier–Stokes equations \cite{AVFursikov_MDGunzburger_LSHou_1998a,MHinze_KKunisch_2004}. Prior to analyzing the numerical behavior of complex optimal control problems governed by fluid dynamics models, it is essential to first study optimization problems constrained by convection–diffusion equations. Numerical studies for unconstrained Dirichlet boundary control problems governed by  convection-diffusion equations have been carried out in \cite{WGong_WHu_MMateos_JSingler_XZhang_YZhang_2018,WHu_MMateos_JSingler_YZhang_2018,HChen_JRSingler_YZhang_2019} using an HDG method, and in \cite{CCorekli_2022} employing a symmetric interior penalty Galerkin approach. 
In the present work, we focus on the numerical analysis of Dirichlet boundary control problems governed by a convection–diffusion equation 
with $L^2$--boundary controls subject to pointwise bounds on the control.  The local discontinuous Galerkin (LDG) method is employed as a discretization technique to address the variational challenges that arise when the control space is taken as $L^2(\Gamma)$ and to leverage the strengths of discontinuous Galerkin methods for convection–diffusion equations. The LDG method is one of several discontinuous Galerkin approaches that have been extensively studied, particularly for convection–diffusion problems, due to its versatility in handling a wide range of applications and advantageous features such as local conservativity and strong locality. \z{Like the HDG method  introduced in \cite{BCockburn_JGopalakrishnan_RLazarov_2009}, the LDG method rewrites higher-order PDEs as systems of first-order equations and employs numerical fluxes to weakly couple discontinuous solutions across element interfaces. However, the two methods differ fundamentally in how these fluxes are constructed and how global coupling is achieved. In LDG, numerical fluxes are defined directly between neighboring elements using information from both sides of each interface, allowing considerable flexibility (e.g., alternating or upwind choices) and making the method particularly robust for a wide range of problems, including convection–diffusion systems. This direct coupling avoids the introduction of additional globally coupled trace variables, as required in HDG, and keeps the formulation conceptually simpler. In contrast, HDG defines its numerical fluxes through a single-valued hybrid variable on element interfaces. This hybrid variable becomes the only globally coupled unknown, and the volumetric element degrees of freedom are recovered locally via static condensation, significantly reducing the size and bandwidth of the global system at the cost of increased formulation complexity.} For studies on the use of discontinuous Galerkin methods in convection-dominated PDEs, we refer the reader to \cite{ABuffa_TJRHughes_GSangalli_2006a,PCastillo_BCockburn_DSchotzau_CSchwab_2002,JCesenek_MFeistauer_2012,YCheng_CWShu_2010,BCockburn_CWShu_1998b}. In the present work, we use the results from \cite{PCastillo_BCockburn_IPergugia_DSchotzau_2000,BCockburn_JGopalakrishnan_FJSayas_2010,JJiang_NJWalkington_YYue_2025} to estimate the approximation errors of the state and adjoint variables. Finally, let us mention that optimal control problems governed by convection–diffusion equations were also investigated in  \cite{PBenner_HYucel_2017,DLeykekhman_MHeinkenschloss_2012a,HYucel_PBenner_2015a,HYucel_MHeinkenschloss_BKarasozen_2013,ZZhou_XYu_NYan_2014a}.


The rest of this paper is organized as follows. \x{In the following section, we discuss the regularity of the solutions and present the optimality system.} Section~\ref{sec:ldg} is devoted to the formulation of the LDG scheme for approximating the Dirichlet boundary control problem. In Section \ref{sec:estimate}, we establish a priori error estimates for the LDG approximation on a convex polygonal domain. These estimates are derived for two distinct discretization strategies used in the optimal control problem: the variational discretization approach and the piecewise linear discretization approach. Section \ref{sec:numeric} presents numerical experiments to support and validate the theoretical findings. Finally, concluding remarks and further discussion are provided in  Section~\ref{sec:conclusion}.

%% file: R2-problem.tex
\section{Regularity and optimality system}\label{sec:regularity}

Throughout this paper, we adopt the standard notation $W^{m,p}(\Omega)$ (see, e.g., \cite{RAAdams_1975}) for Sobolev spaces  equipped with the norm $\| \cdot \|_{m,p,\Omega}$ and seminorm $| \cdot |_{m,p,\Omega}$ for $m \geq 0$ and $1 \leq p \leq \infty$ on an open bounded domain $\Omega$ in $\mathbb{R}^2$ with  boundary $\Gamma=\partial \Omega$. We denote $W^{m,2}(\Omega)$ by $H^m(\Omega)$  with  norm $\| \cdot \|_{m,\Omega}$ and seminorm $| \cdot |_{m,\Omega}$. \x{The spaces of square-integrable functions over $\Omega$ and $\Gamma$ are represented by $L^2(\Omega)$ and $L^2(\Gamma)$, respectively, with norms denoted by $\| \cdot \|_{0,\Omega}$ and $\| \cdot \|_{0,\Gamma}$. The inner products on $L^2 (\Omega)$ and $L^2 (\Gamma)$ are defined by}
\[
(v, w)_{\Omega} = \int_{\Omega} v \, w \, dx 
\qquad \forall v, w \in L^2(\Omega),
\quad
\text{and}
\quad
\langle v, w \rangle_{\Gamma} = \int_{\Gamma} v \, w \, ds 
\qquad \forall v, w \in L^2(\Gamma),
\]
respectively. Note that \x{$H^1_0(\Omega) = \{ v \in H^1(\Omega):  v=0 \; \hbox{on}  \, \Gamma \}$} and \z{$H(\mathrm{div},\Omega) := \left\{ \mathbf{v} \in (L^2(\Omega))^2: \nabla \cdot \mathbf{v} \in L^2(\Omega) \right\}$}. Moreover, $C$ denotes  a generic positive constant independent of the mesh size $h$  and may differ in various estimates. For clarity and convenience, we assume the following conditions on the domain $\Omega$ and the given data, which are  valid throughout the paper:

\begin{assumption}\label{assumption_data}
We assume that $\Omega$ is a convex polygonal domain in $\mathbb{R}^2$ with Lipschitz boundary $\Gamma$ and $\theta \in [\pi/3, \pi)$ represents the largest interior angle of $\Omega$. Diffusion and regularization parameters denoted by $\epsilon$ and $\omega$ are positive constants. The velocity field $\beta \in \big(W^{1,\infty}(\Omega)\big)^2$ complies with incompressibility condition, that is, $\nabla \cdot \beta=0$. The reaction function $\alpha \in L^{\infty}(\Omega)$ satisfies $\alpha(x) \geq 0$. For the source function $f$ and the desired state $y^d$, we assume that $f, y^d \in L^2(\Omega)$.
\end{assumption}

Owing to the low a priori regularity of the control $u\in U:= L^2(\Gamma)$ in the optimal control problem \eqref{p1}, the state equation \eqref{p2} is understood in a very weak sense, i.\,e., we seek a solution $y \in L^2(\Omega)$ satisfying
\begin{align}\label{p4}
-\epsilon (y, \Delta v )_\Omega - (y, \beta \cdot \nabla v )_\Omega + (\alpha y, v)_\Omega = (f,v)_\Omega - \epsilon \langle u, \partial_n v \rangle_{\Gamma}
\end{align}
for all $v \in H_0^1 (\Omega) \cap H^2(\Omega)$ with $\partial_n v := \nabla v \cdot \bn$ and $\bn$ representing the unit outward normal vector to $\Gamma$. It is well known that problem \eqref{p4} for $f\in L^2(\Omega)$ and $u \in L^2(\Gamma)$ admits a unique solution $y\in H^{1/2}(\Omega)$, which  satisfies
  \begin{equation}\label{p5}
    \|y\|_{H^{1/2}(\Omega)} \leq C \big( \|f\|_{0,\Omega} + \|u\|_{0,\Gamma} \big),
\end{equation}
see, e.g., \cite{ECasas_MMateos_JPRaymond_2009, KKunisch_BVexler_2007c,JPfefferer_BVexler_2025}.

Based on the very weak formulation of the state equation \eqref{p4}, the optimal control problem (\ref{p1})-(\ref{p2}) can be stated as follows

\begin{equation}\label{p8}
\underset{(y,u) \in L^2(\Omega) \times U^{ad}}{\hbox{ minimize }} \;  J(y,u) \hbox{ subject to } \eqref{p4}.
\end{equation}
Using the standard arguments in the present setting, it can be shown that problem \eqref{p8} admits a unique solution $(y,u)\in H^1(\Omega)\times (H^{1/2}(\Gamma)\cap \Uad)$, which, together with the unique adjoint state $z\in H^2(\Omega)\cap H^1_0(\Omega)$, weakly solves the following optimality system:
\begin{subequations}\label{p10}
\begin{equation}\label{p10a}
    \begin{split}
        \nabla \cdot ( - \epsilon \nabla y + \beta  y )  + \alpha y =& \, f      \qquad \qquad \quad \hbox{in    }  \Omega,   \\
        y=& \,u    \qquad \qquad \quad \hbox{on    }   \Gamma, 
    \end{split}
\end{equation}     
\begin{equation}\label{p10c}
    \begin{split}
  \nabla \cdot ( - \epsilon \nabla z - \beta  z )  + \alpha z = \,&  y-y^d      \qquad \quad \hbox{in    }  \Omega,   \\
     z=& \,0    \qquad \qquad  \quad \hbox{on    }   \Gamma,
    \end{split}
\end{equation}       
and 
\begin{equation}\label{p10e}
    \langle \omega u - \epsilon \partial_n z, w -u  \rangle_{\Gamma} \geq  0  \qquad \quad \quad \;\;\;\;\;  \forall  w \in \Uad.
\end{equation}
\end{subequations}
Since $z\in H^2(\Omega)\cap H^1_0(\Omega)$, it follows for convex polygonal domains from, e.g., \cite[Theorem~2.5]{JPfefferer_BVexler_2025} that $\partial_n z \in H^{1/2}(\Gamma)$. The variational inequality \eqref{p10e} is equivalent to  $u = \mathcal P_{\Uad} \left(\frac{\epsilon}{\omega} \,  \partial_n z \right) \in H^{1/2}(\Gamma)$; compare, e.g., \cite[Lemma 3.3]{KKunisch_BVexler_2007c}. Here, $\mathcal P_{U^{\textup{ad}}}$ denotes the orthogonal projection in $L^2(\Gamma)$ onto the admissible set $\Uad$. Since \eqref{p10a} admits a unique very weak solution $y(u)$ for every $u\in\Uad$, the reduced cost functional $\widehat J(u):= J(y(u),u)$ is well-defined. Furthermore, $\widehat J'(u) = \omega u-\epsilon \partial_n z$, where $z$ denotes the unique adjoint state satisfying \eqref{p10c} with $y=y(u)$.



To apply the local discontinuous Galerkin (LDG) method for approximating the solution of the Dirichlet boundary control problem \eqref{p1}--\eqref{p2}, we employ a mixed formulation of the optimality system \eqref{p10a}--\eqref{p10e}. This is possible since the optimal control of problem \eqref{p8} satisfies $u\in H^{1/2}(\Gamma)$, and in this case it follows from \eqref{p4} that $\nabla y \in H(\mathrm{div},\Omega)$, and consequently $\nabla y\cdot \bn \in H^{-1/2}(\Gamma)$; compare \cite{WWHu_JGShen_JRSingler_YWZhang_XBZheng_2020, WGong_NYan_2011b}. One now has that with the unique solution  $(y,z,u)$ of the optimality system \eqref{p10a}--\eqref{p10e} the functions $y,\bq,z,\bp,u$, with $\bq:=-\epsilon^{\hF} \nabla y$, $\bp:= \epsilon^{\hF} \nabla z$, $\bq, \bp \in H(\mathrm{div},\Omega)$, form the unique weak solution of the mixed system given by
\begin{subequations}\label{p13}
\begin{align}\label{p13a}
\begin{aligned}
\nabla \cdot (\beta\, y + \epsilon^{\hF} \mathbf{q}) + \alpha\, y &= f 
    &&\text{in }\Omega,\\[2mm]
\mathbf{q} &= -\,\epsilon^{\hF} \nabla y
    &&\text{in }\Omega,\\[2mm]
y &= u 
    &&\text{on }\Gamma,
\end{aligned}
\end{align}
\begin{align}\label{p13b}
\begin{aligned}
\nabla \cdot (-\beta\, z - \epsilon^{\hF} \mathbf{p}) + \alpha\, z &= y - y^{d}
    &&\text{in }\Omega,\\[2mm]
\mathbf{p} &= \epsilon^{\hF} \nabla z
    &&\text{in }\Omega,\\[2mm]
z &= 0
    &&\text{on }\Gamma,
\end{aligned}
\end{align}
and
\begin{align}\label{p13c}
\langle\, \omega u - \epsilon^{\hF} \mathbf{p}\cdot \bn,\, w - u\, \rangle_{\Gamma} \;\ge\; 0,
    \qquad \forall w \in \Uad.
\end{align}
\end{subequations}


To ensure the approximation properties of the LDG method, specific regularity assumptions on the involved states and fluxes are required. For the solution of our optimal Dirichlet boundary control problem on polygonal domains, a detailed study can be found in \cite{TApel_MMateos_JPfefferer_ARosch_2015}. In this respect, we obtain the following theorem from e.g., \cite[Theorem~2.10]{WHu_MMateos_JSingler_YZhang_2018}, which is also valid in our control constrained case; compare the investigations in \cite[Section 4]{TApel_MMateos_JPfefferer_ARosch_2015} related to the convex case.
\begin{theorem}\label{thm:higher_reg}
   Suppose that Assumption~\ref{assumption_data} is satisfied  and that $y^d \in H^{t}(\Omega)$ with $0 \leq t < 1$. Then, for $s\in [\frac{1}{2}, \min\{\frac{1}{2}+t,\pi/\theta -1/2\})$, for the solution of the optimality system \eqref{p13a}--\eqref{p13c}, we have 
  \begin{equation*}
     u \in H^s(\Gamma), \;\; y \in H^{s+1/2}(\Omega), \;\;  \bq \in H^{s-1/2}(\Omega)^2\cap H(\mathrm{div},\Omega),\end{equation*}
as well as
\begin{equation*}
z\in H^{s+3/2}(\Omega)\cap H^1_0(\Omega) \;\; \text{and} \;\; \bp\in H^{s+1/2}(\Omega)^2\cap H(\mathrm{div},\Omega).
  \end{equation*}  
\end{theorem}


%% file: R2-DG.tex
\section{Local discontinuous Galerkin formulation}\label{sec:ldg}

Since we assume that the domain $\Omega$ is polygonal, its boundary is exactly represented by triangle edges. We denote $\{ \cT_h\}_h$ as a family of shape-regular simplicial triangulations of $\Omega$. Each mesh $\cT_h$ consists of closed triangles such that $\overline{\Omega} = \bigcup_{K \in \cT_h} K$ holds. We assume that the mesh is regular in the following sense: for different triangles
$K_i, K_j \in \cT_h$, $i \not= j$, the intersection  $K_i \cap K_j$ is either empty, a vertex, or an edge, i.e., hanging nodes are not allowed. The diameter of an element $K$ and the length of an edge $E$ are denoted by $h_{K}$ and $h_E$, respectively, with $h = \max \limits_{K \in \cT_h}  h_K$.
The  set of all edges $\cE_h$ is split into the set $\cE^0_h$ of interior edges and the set $\cE^{\partial}_h$ of boundary edges, so that $\cE_h=\cE^{0}_h  \cup \cE^{\partial}_h$. For the outward unit normal  vector  $\bn$ on $\Gamma$, the inflow and outflow parts of $\Gamma$ are denoted by $\Gamma^-$ and $\Gamma^+$, respectively,
\[
          \Gamma^- = \set{x \in \Gamma}{\beta\cdot \bn < 0}, \quad
          \Gamma^+ = \set{x \in \Gamma}{\beta \cdot \bn \geq 0}.
\]
Similarly, the boundaries of an element $K \in \cT_h$ can be categorized into inflow and outflow boundaries: $\partial K^-$ and $\partial K^+$. 

Let $E$ be a common edge for two elements $K$ and $K^e$. For a piecewise continuous scalar function $y$, there are two traces of $y$ along $E$, denoted by $y|_E$ from inside $K$ and $y^e|_E$ from inside $K^e$. Then, the jump and average of $y$ across the edge $E$ are defined by:
\begin{align*}
    \jump{y}&=y|_E\bn_{K}+y^e|_E\bn_{K^e}, \quad
    \average{y}=\frac{1}{2}\big( y|_E+y^e|_E \big),
\end{align*}
where $\bn_K$ (resp. $\bn_{K^e}$) denotes the outward unit  normal vector to $\partial K$ (resp. $\partial K^e$). Similarly, for a piecewise continuous vector field $\mathbf{q}$, the jump and average across an edge $E$ are defined, respectively, by
\begin{align*}
           \jump{\mathbf{q}}&=\mathbf{q}|_E \cdot \bn_{K} + \mathbf{q}^e|_E \cdot \bn_{K^e}, \quad
          \average{\mathbf{q}}=\frac{1}{2}\big(\mathbf{q}|_E+\mathbf{q}^e|_E \big).
\end{align*}
For a boundary edge \zz{$E \in \cE^{\partial}_h$ with $E \subset K \cap \Gamma$}, we set $\average{\mathbf{q}}=\mathbf{q}$ and $\jump{y}=y\bn$,
where $\bn$ is the outward unit normal  vector on $\Gamma$. Note that the jump in $y$ is vector-valued, whereas the jump in $\mathbf{q}$ is  scalar-valued which only involves the normal component of $\mathbf{q}$.

To adapt the continuous mixed formulation to the LDG setting, we note that the weak solution of \eqref{p13a} also satisfies  the following elementwise system on each element $K \in \cT_h$
\begin{subequations}\label{p14}
\begin{eqnarray}
-(\epsilon^\hF \bq + \beta y, \nabla v)_K + (\alpha y, v)_K + \langle ( \epsilon^\hF \bq + \beta y ) \cdot \bn, v \rangle_{\partial K} &=& (f,v)_K   \quad \;\; \forall  v \in V, \label{p14a} \\
( \bq, \mathbf{r} )_K - ( \epsilon^\hF y,  \nabla \cdot  \mathbf{ r})_K  +\langle \epsilon^\hF u,  \mathbf{r} \cdot \bn \rangle_{\partial K} &=&0 \qquad \qquad \forall \mathbf{r} \in \mathbf{W}, \label{p14b}
\end{eqnarray}
\end{subequations}
where
\begin{equation*}
\begin{split}
    \mathbf{W} :=& \, \set{\mathbf{w} \in \left( L^2(\Omega) \right)^2}{\mathbf{w}_K \in \left( H^1(K) \right)^2, \;\; \forall K \in \cT_h}, \\
    V :=& \, \set{v\in L^2(\Omega)}{v_K \in H^1(K), \;\; \forall K \in \cT_h}.
\end{split}
\end{equation*}

Next, we seek to approximate the state solution $(y, \bq)$  by functions $(y_h, \bq_h)$ in the following finite element spaces \zz{$V_h \times \bW_h   \subset V \times \bW $}:
\begin{subequations}\label{DG1}
\begin{align}
    \mathbf{W}_h &= \set{\mathbf w \in \big(L^2(\Omega)\big)^2}{ \mathbf w\mid_{K}\in \big( \mathbb{S}^1(K) \big)^2, \quad \forall K \in \cT_h}, \\[4pt]
    V_h &= \set{v \in L^2(\Omega)}{ v \mid_{K}\in \mathbb{S}^1(K), \quad \forall K \in \cT_h}, \\[4pt]
    U_h &= \set{u \in L^2(\Gamma)}{ u \mid_{E}\in \mathbb{S}^1(E), \quad \forall E \in \cE_h^{\partial}},
\end{align}
\end{subequations}
where $\mathbb{S}^1(K)$ (resp. $\mathbb{S}^1(E)$) denotes the local finite element space, which consists of linear polynomials on each element $K$ (resp. on $E$). Then, for all $(v, \br) \in  V_h \times \bW_h$ the approximate solution $(y_h,\bq_h)$ of the state solution $(y,\bq)$ satisfies
\begin{subequations}\label{DG2}
\begin{eqnarray}
-(\epsilon^\hF \bq_h + \beta y_h, \nabla v)_{K} + (\alpha y_h, v)_{K} + \langle ( \epsilon^\hF \widehat{\bq}_h + \beta \widetilde{y}_h ) \cdot \bn, v \rangle_{\partial K} &=& (f,v)_{K}, \\
( \bq_h, \br )_{K} - (\epsilon^\hF y_h, \nabla \cdot  \br)_{K} + \langle \epsilon^\hF \widehat{y}_h,  \br \cdot \bn \rangle_{ \partial K} &=&0.
\end{eqnarray}
\end{subequations}
The numerical fluxes, denoted by $\widehat{\bq}_h, \widetilde{y}_h$, and  $\widehat{y}_h$,  must be chosen appropriately  to ensure the stability of the method and to improve its accuracy.

We are now ready to introduce the expressions that define the numerical fluxes. The numerical traces of $y_h$ associated with the diffusion and convection terms are characterized, respectively, by
\begin{equation}\label{DG3}
\widehat{y}_h = \left\{
                  \begin{array}{ll}
                    \average{y_h} + \mathbf{C}_{12} \cdot \jump{y_h}, & E \in \cE^{0}_h, \\
                    u_h, & E \in \cE^{\partial}_h,
                  \end{array}
                \right.
\quad \hbox{and} \quad
\widetilde{y}_h =\left\{
                 \begin{array}{ll}
                   u_h, & \zz{E \in \cE_h^\partial, \; E \subset \Gamma^{-}}, \\
                   \average{y_h} + \mathbf{D}_{11} \cdot \jump{y_h}, & E \in \cE^{0}_h, \\
                   y_h, & \zz{E \in \cE_h^\partial, \; E \subset \Gamma^+}.
                 \end{array}
               \right.
\end{equation}
From \eqref{DG3}, it follows that the numerical trace $\widetilde{y}_h$ aligns with the traditional upwinding trace. Moreover, the numerical flux  $\widehat{\bq}_h $ is given by
\begin{eqnarray}\label{DG4}
\widehat{\bq}_h &=& \left\{
                  \begin{array}{ll}
                    \average{\bq_h} - C_{11}\jump{y_h} - \mathbf{C}_{12} \jump{\bq_h}, & E \in \cE^{0}_h, \\
                    \bq_h  - C_{11} (y_h -u_h) \bn, & E \in \cE^{\partial}_h.
                  \end{array}
                \right.
\end{eqnarray}
Here, we set $C_{11} = \epsilon/h_E$ for each $E \in \cE_h$,  choose $\mathbf{C}_{12}$ such that $\mathbf{C}_{12} \cdot \bn = \frac{1}{2} \sign \big( \bn \cdot \mathbf v \big)$ for a nonzero arbitrary \x{but fixed} vector $\mathbf{v}$, and define the vector-valued function $\mathbf{D}_{11}$ by
\[
\mathbf{D}_{11} \cdot \bn = \frac{1}{2} \sign \big( \bn \cdot \beta \big).
\]
\x{Note that the auxiliary variable $\mathbf{v}$ need not be associated with the convective velocity $\boldsymbol{\beta}$; it may be chosen independently without affecting the 
convergence properties of the method. Nevertheless, when $\boldsymbol{\beta}$ is not 
identically zero on any element $K \in \mathcal{T}_h$, it is natural and often convenient 
to relate $\mathbf{v}$ to $\boldsymbol{\beta}$.} For a more detailed discussion, we refer the reader to \cite{PCastillo_BCockburn_IPergugia_DSchotzau_2000,BCockburn_JGopalakrishnan_FJSayas_2010,BCockburn_GKanschat_DSchotzau_2004a,ZZhou_XYu_NYan_2014a}.

Inserting the numerical fluxes in \eqref{DG3}--\eqref{DG4} into \eqref{DG2} and summing over all elements, we obtain
\begin{align*}
&-\sum\limits_{K \in \cT_h} \int_{K} (\epsilon^{\hF} \bq_h + \beta y_h) \cdot \nabla v \, dx  
+ \sum\limits_{E \in \cE_h^0} \int_E \epsilon^{\hF} 
\left( \average{\bq_h} - C_{11} \jump{y_h} - \mathbf{C}_{12} \jump{\bq_h} \right) 
\cdot \jump{v} \, ds \\
&\quad + \sum\limits_{K \in \cT_h} \int_{K} \alpha y_h \, v \, dx 
+ \sum\limits_{E \in \cE_h^0} \int_E 
\big( \average{y_h} + \mathbf{D}_{11} \cdot \jump{y_h} \big) 
\beta \cdot \jump{v} \, ds 
+ \sum\limits_{\zz{\substack{E \in \cE_h^\partial \\ E \subset \Gamma^+}}} \int_E (\bn \cdot \beta) y_h \, v \, ds \\
&\quad + \sum\limits_{E \in \cE_h^{\partial}} \int_E 
\epsilon^{\hF} \big( \bq_h \cdot \bn - C_{11} y_h \big) v \, ds 
= \int_{\Omega} f \, v \, dx 
- \sum\limits_{E \in \cE_h^{\partial}} \int_E \epsilon^{\hF} C_{11} u_h \, v \, ds 
+ \sum\limits_{\zz{\substack{E \in \cE_h^\partial \\ E \subset \Gamma^-}}} \int_E |\beta \cdot \bn| u_h \, v \, ds,
\end{align*}

and

\begin{align*}
&\int_{\Omega} \bq_h \cdot \br \, dx  
- \sum\limits_{K \in \cT_h} \int_{K} \epsilon^{\hF} y_h \, \nabla \cdot \br \, dx  
+ \sum\limits_{E \in \cE_h^0} \int_E \epsilon^{\hF} 
\big( \average{y_h} + \mathbf{C}_{12} \cdot \jump{y_h} \big) 
\jump{\br} \, ds \\
&\quad = -\sum\limits_{E \in \cE_h^{\partial}} \int_E \epsilon^{\hF} u_h \, \br \cdot \bn \, ds.
\end{align*}
It is convenient to introduce the following bi(linear) forms:
\begin{align*}
a_h(\bq, \br) :=& \int_{\Omega} \bq \cdot \br \, dx, \\[0.3em]
b_h(y, \br) :=& -\sum\limits_{K \in \cT_h} \int_{K} \epsilon^{\hF} y \, \nabla \cdot \br \, dx  
+ \sum\limits_{E \in \cE_h^0} \int_E \epsilon^{\hF} 
\big( \average{y} + \mathbf{C}_{12} \cdot \jump{y} \big) \jump{\br} \, ds, \\[0.3em]
c_h(y, v) :=& \sum\limits_{K \in \cT_h} \int_{K} 
\big( \alpha y \, v - y \, \beta \cdot \nabla v \big) \, dx  
+ \sum\limits_{E \in \cE_h^0} \int_E 
\big( \average{y} + \mathbf{D}_{11} \cdot \jump{y} \big) 
\beta \cdot \jump{v} \, ds \\
& - \sum\limits_{E \in \cE_h^0} \int_E \epsilon^{\hF} C_{11} \jump{y} \cdot \jump{v} \, ds 
+ \sum\limits_{\zz{\substack{E \in \cE_h^\partial \\E \subset \Gamma^+}}} \int_E (\bn \cdot \beta) y \, v \, ds 
- \sum\limits_{E \in \cE_h^{\partial}} \int_E \epsilon^{\hF} C_{11} y \, v \, ds, \\
m_{h,1}(u, \br) :=& -\sum\limits_{E \in \cE_h^{\partial}} \int_E \epsilon^{\hF} u \, \br \cdot \bn \, ds, \\[0.3em]
m_{h,2}(u, v) :=&  -\sum\limits_{E \in \cE_h^{\partial}} \int_E \epsilon^{\hF} C_{11} u \, v \, ds  
+ \sum\limits_{\zz{\substack{E \in \cE_h^\partial \\ E \subset \Gamma^-}}} \int_E |\beta \cdot \bn| u \, v \, ds, \\[0.3em]
F(v) :=& \int_{\Omega} f \, v \, dx.
\end{align*}
By applying integration by parts to the first term in $b_h(\cdot,\cdot)$, we obtain
\begin{align*}
b_h(y, \br) 
=& -\sum\limits_{K \in \cT_h} \int_{K} \epsilon^{\hF} y \, \nabla \cdot \br \, dx  
   + \sum\limits_{E \in \cE_h^0} \int_E \epsilon^{\hF} 
   \big( \average{y} + \mathbf{C}_{12} \cdot \jump{y} \big) 
   \jump{\br} \, ds \\
=&  \sum\limits_{K \in \cT_h} \int_{K} \epsilon^{\hF} 
   \nabla y \cdot \br \, dx 
   - \sum\limits_{K \in \cT_h} \int_{\partial K} \epsilon^{\hF} 
   y \, \br \cdot \bn \, ds \\
&\quad + \sum\limits_{E \in \cE_h^0} \int_E \epsilon^{\hF} 
   \big( \average{y} + \mathbf{C}_{12} \cdot \jump{y} \big) 
   \jump{\br} \, ds.
\end{align*}
Then, by the following straightforward computation
\[
\sum\limits_{K \in \cT_h} \int_{\partial K} \epsilon^{\hF} y \, \br \cdot \bn \, ds 
= \sum\limits_{E \in \cE_h^0 \cup \cE_h^{\partial}} 
\int_E \average{\br} \cdot \jump{\epsilon^{\hF} y} \, ds 
+ \sum\limits_{E \in \cE_h^0} 
\int_E \jump{\br} \cdot \average{\epsilon^{\hF} y} \, ds,
\]
we get
\begin{align*}
b_h(y, \br) 
=&  \sum\limits_{K \in \cT_h} \int_{K} \epsilon^{\hF} \nabla y \cdot \br \, dx  
- \sum\limits_{E \in \cE_h^0} \int_E \epsilon^{\hF} 
\big( \average{\br} - \mathbf{C}_{12} \jump{\br} \big) 
\cdot \jump{y} \, ds \\
&\quad - \sum\limits_{E \in \cE_h^{\partial}} 
\int_E \epsilon^{\hF} y \, \br \cdot \bn \, ds.
\end{align*}

For given right-hand side $f$ and boundary condition $u$, the LDG approximation of the unique solution $(y,\bq)$ to \eqref{p13a} is given by $(y_h,\bq_h) \in V_h\times \bW_h$ satisfying
\begin{subequations}\label{LDG-approx}
\begin{align}
    a_h(\bq_h, \br) + b_h(y_h, \br) &= m_{h,1}(u, \br) 
    && \forall \br \in \bW_h, \\
   -b_h(v, \bq_h) + c_h(y_h, v) &= m_{h,2}(u, v) + F(v) 
    && \forall v \in V_h.
\end{align}
\end{subequations}
For given $u\in L^2(\Gamma)$ and $f\in L^2(\Omega)$, this system admits a unique solution, which can be established using arguments similar to those in \cite[Proposition 2.1]{PCastillo_BCockburn_IPergugia_DSchotzau_2000}.

Hence, the LDG approximation scheme of the Dirichlet boundary control problem 
\eqref{p1}--\eqref{p2} reads as 
\begin{equation}\label{DG6}
    \underset{u_h \in \Uadh, \, (y_h, \bq_h) \in V_h \times \bW_h}{\hbox{minimize}} \;
    J(y_h,u_h) = \frac{1}{2}\|y_h - y^d\|^{2}_{0,\Omega} 
    + \frac{\omega}{2} \|u_h\|^{2}_{0,\Gamma}
\end{equation}
subject to 
\begin{subequations}\label{DG5}
\begin{align}
    a_h(\bq_h, \br) + b_h(y_h, \br) &= m_{h,1}(u_h, \br) 
    && \forall \br \in \bW_h, \\
   -b_h(v, \bq_h) + c_h(y_h, v) &= m_{h,2}(u_h, v) + F(v) 
    && \forall v \in V_h.
\end{align}
\end{subequations}

\noindent Here, the set of admissible controls is chosen either as $\Uadh:= \Uad$, which is the case of variational discretization \cite{MHinze_2005a}, or as $\Uadh:=\Uad \cap U_h$, which we refer to as the case of full discretization. Since $\Uadh$ is nonempty,  closed, and convex and $J$ is uniformly convex, problem \eqref{DG6}-\eqref{DG5} admits a unique solution $((y_h,\bq_h),u_h)\in (V_h\times \bW_h)\times \Uadh$.
Moreover, it follows from standard arguments that $((y_h,\bq_h),u_h)$ is the unique solution of the optimal control problem \eqref{DG6}-\eqref{DG5} if and only if a pair $(\bp_h,z_h)\in \bW_h\times V_h$ of discrete adjoint flux and discrete adjoint potential exists, which together with $((y_h,\bq_h),u_h,(\bp_h,z_h))$ solves the following discrete optimality conditions
\begin{subequations}\label{DG8}
\begin{align}
a_h(\bq_h, \br) + b_h(y_h, \br) 
&= m_{h,1}(u_h, \br)
&& \forall \br \in \bW_h, \label{DG8a} \\[0.3em] 
-b_h(v, \bq_h) + c_h(y_h, v) 
&= m_{h,2}(u_h, v) + F(v)
&& \forall v \in V_h, \label{DG8b} \\[0.3em] 
a_h(\bp_h, \psi) - b_h(z_h, \psi) 
&= 0 
&& \forall \psi \in \bW_h, \label{DG8c} \\[0.3em] 
b_h(\phi, \bp_h) + c_h(\phi, z_h) 
&= (y_h - y^d, \phi)_\Omega
&& \forall \phi \in V_h, \label{DG8d} \\[0.3em] 
\langle \omega u_h + m'_{1,h}(u_h, \bp_h) + m'_{2,h}(u_h, z_h), \, w - u_h \rangle_\Gamma 
&\ge 0
&& \forall w \in \Uadh. \label{DG8e}
\end{align}
\end{subequations}
The variational inequality in \eqref{DG8e} in terms of \zz{$\bp_h$} and $z_h$ reads
\begin{equation}
\langle \omega u_h - \epsilon^{\hF} \bp_h \cdot \bn 
+ \kappa_z z_h, 
\, w - u_h \rangle_{\Gamma} \ge 0
\qquad \forall w \in \Uadh,
\end{equation}
where $\kappa_z =  -\epsilon^{\hF} C_{11} + \chi_{\Gamma^-} |\beta \cdot \bn|$  \x{with  $\chi_{\Gamma^-}$ denoting the indicator function of $\Gamma^-$.}

If the optimal solution $\big( y, \bq, z, \bp, u \big) 
\in V \times \bW \times V \times \bW \times \Uad$ is smooth enough, it also satisfies the system
\begin{subequations}\label{sys:optimality_cont}
\begin{align}
a_h(\bq, \br) + b_h(y, \br) 
&= m_{h,1}(u, \br)
&& \forall \br \in \bW,  \\[0.3em]
-b_h(v, \bq) + c_h(y, v) 
&= m_{h,2}(u, v) + F(v)
&& \forall v \in V, \\[0.3em]
a_h(\bp, \psi) - b_h(z, \psi) 
&= 0
&& \forall \psi \in \bW,  \\[0.3em]
b_h(\phi, \bp) + c_h(\phi, z) 
&= (y - y^d, \phi)_\Omega
&& \forall \phi \in V, \\[0.3em]
\langle \omega u + m'_{1,h}(u, \bp) + m'_{2,h}(u, z), \, w - u \rangle_\Gamma 
&\ge 0
&& \forall w \in \Uad.
\end{align}
\end{subequations}

%% file: R2-estimate_alternative.tex
\section{Error estimates}\label{sec:estimate}

In this section, we present error estimates for the LDG approximation of the Dirichlet boundary control problem posed on a convex polygonal domain. The analysis is carried out for two different discretization strategies applied to the optimal control problem: the variational discretization approach and the piecewise linear discretization approach.

Let us introduce the \z{(linear)} control-to-state operator
$
    S : L^2(\Gamma) \to H^{\hF}(\Omega),
$
defined by $u \mapsto S(u) := y$, where for a given $u$ the state $y$ satisfies \eqref{p4}. Using this, the reduced cost functional can be written as
\[
    \widehat J : U \to \mathbb{R}, 
    \qquad 
    \widehat J(u) = J(S u, u).
\]
Hence, the optimal control problem (\ref{p1})-(\ref{p2}) can be reformulated as
\begin{equation}\label{eqn:reduced_problem}
    \min_{u \in \Uad} \widehat J(u).
\end{equation}
Following standard arguments in \cite{JLLions_1971}, the existence and uniqueness of an optimal solution $u \in \Uad$ can be established by invoking the continuity and strict convexity of $\widehat J$. For a control $u \in U$ and a direction $\delta u \in U$, the directional derivatives are given by 
\begin{equation}\label{continuous_derivative}
 \widehat J'(u)(\delta u) = (y - y^d, y(\delta u))_\Omega + \omega \langle u, \delta u \rangle_{\Gamma} = \langle\, \omega u - \epsilon^{\hF} \mathbf{p}\cdot \bn,\, \delta u\, \rangle_{\Gamma},
\end{equation}
where $y$ is the state associated with the control $u$, $y(\delta u)$ denotes the state assigned to the control $\delta u$, and $\bp$ is the adjoint  flux associated with the observation $y-y^d$. The necessary  optimality condition of the reduced problem \eqref{eqn:reduced_problem} is given as a variational inequality:
\begin{equation}\label{reduced_cont_variational}
   \widehat J'(u)(\delta u - u) \geq 0 \quad \forall \delta u \in \Uad.
\end{equation}

To discretize the reduced optimal control problem~\eqref{eqn:reduced_problem}, 
we define the discrete reduced cost functional as
\[
    \widehat{J}_h : U \to \mathbb{R}, 
    \qquad 
    \widehat{J}_h(u) = J(S_h u,u),
\]
where $S_h : u \mapsto y_h(u)$ represents the discrete solution operator, which, according to \cite[Proposition 2.1]{PCastillo_BCockburn_IPergugia_DSchotzau_2000}, is well-defined. 
The corresponding discretized reduced problem is then given by
\begin{equation}\label{eqn:discrete_reduced_problem}
    \min_{u \in \Uadh} \; \widehat{J}_h(u).
\end{equation}
Analogous to the continuous case, the discrete reduced cost functional 
$\widehat{J}_h$ is quadratic. Its directional  derivative is given by
\begin{equation}\label{discrete_derivative}
    \widehat{J}_h'(u)(\delta u) 
    = (y_h(u) - y^d,\, y_h(\delta u))_{\Omega}
      + \omega \langle u,\, \delta u \rangle_{\Gamma} = \langle \omega u - \epsilon^{\hF} \bp_h \cdot \bn 
+ \kappa_z z_h, 
\, \delta u \rangle_{\Gamma},
\end{equation}
where $y_h(u)$, $y_h(\delta u)$ \x{denote} discrete states associated with $u$ and $\delta u$, respectively, and $\bp_h$ and $z_h$ denote the adjoint flux and potential associated with $y_h(u)-y^d$.

In this work, we consider two approaches for the discretization of the control variable: 
the \emph{variational discretization} proposed in~\cite{MHinze_2005a}, 
and the \emph{piecewise linear discretization}. 
For the linear discretization, we employ the discrete space defined in~\eqref{DG1}, 
while in the case of variational discretization, the discrete control space is chosen as 
$U_h = U = L^2(\Gamma)$. 
In both settings, the discrete admissible control set is given by 
$\Uadh = U_h \cap \Uad$. 
Furthermore, for the variational discretization approach, the discrete optimal control $u_h$ is characterized by
\begin{equation}\label{reduced_disc_var_variational}
    \widehat{J}_h'(u_h)(\delta u - u_h) \ge 0,
    \qquad \forall \delta u \in \Uad,
\end{equation}
whereas for the piecewise linear discretization approach, the discrete optimal control $u_h$ satisfies
\begin{equation}\label{reduced_disc_variational}
    \widehat{J}_h'(u_h)(\delta u_h - u_h) \ge 0,
    \qquad \forall \delta u_h \in \Uadh.
\end{equation}
In the analysis, we require an interpolation/projection operator that preserves admissibility; however, the $L^2$-projection does not, in general, preserve the admissibility of functions in $\Uad$. Hence, we employ the quasi-interpolation operator \zz{$\pi_h : L^1(\Gamma) \to U_h$} inspired by \cite{CCarstensen_1999}, and used in the context of Dirichlet boundary control, see, e.g.,  \cite{JPfefferer_BVexler_2025}. The operator $\pi_h$ is defined as
\begin{equation}\label{def:quasi_interpolation}
    \pi_h u = \sum_{n \in \mathcal{N}_h^\partial} \frac{(u, \xi_n)}{(1, \xi_n)} \, \xi_n,
\end{equation}
where $\mathcal{N}_h^\partial$ denotes the set of all mesh nodes of $\mathcal{T}_h$ lying on the boundary $\Gamma$, and $\xi_n \in U_h$ represents the local (nodal) basis function associated with node $n$ on the boundary. By construction, it follows that $\pi_h u \in \Uadh$ for every $u \in \Uad$. For $u \in H^{s}(\Gamma)$  with $0 \leq s \leq 1$, the following estimates  also hold \cite[Lemma~5.5]{JPfefferer_BVexler_2025}
\begin{subequations} \label{ineq:quasi_interpolation}
\begin{align}
  \| u - \pi_h u \|_{0,\Gamma} &\leq C h^{s} \| u \|_{s,\Gamma}, \label{ineq:quasi_interpolationa} \\[4pt]
  \| u - \pi_h u \|_{-s,\Gamma} &\leq C h^{2s} \| u \|_{s,\Gamma}. \label{ineq:quasi_interpolationb}
\end{align}
\end{subequations}

In the numerical error analysis, it will be convenient to use the adjoint-based representations of the respective derivatives $\widehat J'$ and $\widehat J_h'$. \x{In particular, we write
\begin{equation}\label{adj-deriv-z}
\widehat J'(u)(\delta u)  = \langle\, \omega u - \epsilon^{\hF} \mathbf{p}\cdot \bn + \kappa_z z,\, \delta u\, \rangle_{\Gamma}.
\end{equation}
This representation is valid since $z=0$ on $\Gamma$.}

To facilitate the analysis, we also introduce the following auxiliary problem: find $(z_h(y), \bp_h(y)) \in V_h \times \bW_h$ such that
\begin{subequations}\label{aux2}
\begin{align}
a_h(\bp_h(y), \psi) - b_h(z_h(y), \psi) &= 0,
    && \forall \psi \in \bW_h, \\[4pt]
b_h(\phi, \bp_h(y)) + c_h(\phi, z_h(y)) &= (y - y^d, \phi)_\Omega,
    && \forall \phi \in V_h.
\end{align}
\end{subequations}
Here, $(z_h(y), \bp_h(y))$ denotes the LDG approximation of $(z, \bp)$. 

We begin by establishing an error estimate for the variational discretization.

\begin{theorem}\label{thm:nodisc_adjoint_representation}
Let $u$ be the solution of~\eqref{eqn:reduced_problem}, and let 
$u_h$ be the corresponding discrete control obtained with the choice $U_h = U = L^2(\Gamma)$. 
Moreover, let $y$ and $y_h$ denote the corresponding continuous and discrete state solutions, respectively. 
Then, 
\begin{equation}\label{vge}
    \omega \| u - u_h \|_{0,\Gamma}^2  + \|y-y_h\|_{0,\Omega}^2 \le \|y-y_h(u)\|_{0,\Omega}^2 + C(\omega) 
     \big( \epsilon \|(\bp_h(u) - \bp)\cdot\bn\|_{0,\Gamma}^2 
         + \kappa_z^2\|z_h(u) - z\|_{0,\Gamma}^2 \big),
\end{equation}
where the positive constant $C$ is independent of the mesh size~$h$, and where $\bp_h(u):=\bp_h(y(u),\bq(u))$ and $z_h(u):=z_h(y(u),\bq(u))$ are local discontinuous Galerkin approximations of $\bp$ and $z$ associated with the continuous data $u, y(u)$, and $\bq(u)$, respectively, i.e., $\big(z_h(u),\bp_h(u)\big)$ solves \eqref{DG8c}-\eqref{DG8d} with right-hand side $y-y^d$.

\end{theorem}
\begin{proof}
By taking $\delta u = u_h$ in \eqref{reduced_cont_variational} and $\delta u = u$ in \eqref{reduced_disc_var_variational}, we obtain
\[
\langle \widehat{J}'(u_h) - \widehat{J}'(u),\, u - u_h \rangle \ge 0,
\]
that is, since $z_{|_\Gamma} = 0$,
\[
\langle \omega (u_h - u) - \epsilon^{\hF} (\bp_h - \bp) \cdot \bn + \kappa_z (z_h - z),\, u - u_h \rangle_{\Gamma} \ge 0.
\]
Recalling the definition of $m_{h,1}$ and $m_{h,2}$, this implies that
\begin{align}\label{eqn:v1}
\omega \|u - u_h\|^2_{0,\Gamma}
&\le m_{h,1}(u - u_h,\, \bp_h - \bp) + m_{h,2}(u - u_h,\, z_h - z) \nonumber \\[4pt]
&= m_{h,1}(u - u_h,\, \bp_h - \bp_h(u)) + m_{h,2}(u - u_h,\, z_h - z_h(u)) \nonumber\\
&\quad + m_{h,1}(u - u_h,\, \bp_h(u) - \bp) + m_{h,2}(u - u_h,\, z_h(u) - z) \nonumber \\[4pt]
&=: I_1 + I_2.
\end{align}

To derive a bound for $I_1$ in \eqref{eqn:v1}, we make use of the optimality conditions  \eqref{DG8a}-\eqref{DG8d} together with Young's inequality
\begin{align}\label{eqn:v2}
I_1 
&= m_{h,1}(u - u_h,\, \bp_h - \bp_h(u)) 
   + m_{h,2}(u - u_h,\, z_h - z_h(u)) \nonumber \\[4pt]
&= a_h(\bq_h(u) - \bq_h,\, \bp_h - \bp_h(u)) 
   + b_h(y_h(u) - y_h,\, \bp_h - \bp_h(u)) \nonumber \\
&\quad - b_h(z_h - z_h(u),\, \bq_h(u) - \bq_h) 
   + c_h(y_h(u) - y_h,\, z_h - z_h(u)) \nonumber \\[4pt]
&= (y_h - y,\, y_h(u)- y_h)_{\Omega}  = (y_h - y,\, y_h(u) - y + y - y_h)_{\Omega} \nonumber \\[4pt]
&\le - \frac{1}{2}\|y - y_h\|^2_{0,\Omega} 
   + \frac{1}{2}\|y_h(u) - y\|^2_{0,\Omega},
\end{align}
where $y_h(u)$ denotes the LDG approximation of the state variable corresponding to the Dirichlet boundary condition $u$.

\x{Using Young's inequality, for $I_2$ we get}
\begin{align}\label{eqn:v3}
I_2 
&= m_{h,1}(u - u_h,\, \bp_h(u) - \bp) 
  + m_{h,2}(u - u_h,\, z_h(u) - z) \nonumber \\[4pt]
&\le \frac{\omega}{2} \|u - u_h\|_{0,\Gamma}^2 
   + C(\omega) 
     \big( \epsilon \|(\bp_h(u) - \bp)\cdot\bn \|_{0,\Gamma}^2 
         + \kappa_z^2\|z_h(u) - z\|_{0,\Gamma}^2 \big). 
\end{align}
Combining \eqref{eqn:v2} and \eqref{eqn:v3} in \eqref{eqn:v1}, we obtain the estimate in \eqref{vge}.
\end{proof}

For the fully discrete approach we have

\begin{theorem}\label{thm:fully-discrete}
Let $u$ be the solution of~\eqref{eqn:reduced_problem}, and let
$u_h$ be the corresponding discrete control obtained with the choice of $U_h$ defined in~\eqref{DG1}.
Further, let $y$ and $y_h$ denote the corresponding \x{continuous and discrete states}, respectively. 
Then, 
\begin{multline}\label{gge}
    \omega \| u - u_h \|_{0,\Gamma}^2  + \|y-y_h\|_{0,\Omega}^2 \le \|y-y_h(u)\|_{0,\Omega}^2 + C(\omega) 
     \big( \epsilon\|(\bp_h(u) - \bp)\cdot \bn\|_{0,\Gamma}^2 
         + \kappa_z^2\|z_h(u) - z\|_{0,\Gamma}^2 \big)  \\ + C(\omega) \|\pi_hu-u\|_{0,\Gamma}^2 + \big(\epsilon^{1/2}\|(\bp_h - \bp(y_h))\cdot \bn\|_{0,\Gamma}+ \kappa_z\|z_h - z(y_h)\|_{0,\Gamma}\big)\|\pi_hu-u\|_{0,\Gamma}  \\ + C \big(\|u\|_{s,\Gamma}, \epsilon^{1/2}\|\bp(u)\cdot\bn\|_{s,\Gamma}\big)\,
       \|\pi_h u - u\|_{-s,\Gamma},
\end{multline}
where the triplet $\big(y_h(u), z_h(u),\bp_h(u)\big)$ denotes the local discontinuous Galerkin approximations of $(y,z,\bp)$ for a given $u\in H^s(\Gamma)$  with $0\le s\le 1$, i.e., $\big(y_h(u), z_h(u),\bp_h(u)\big)$ solves \eqref{DG8a}-\eqref{DG8d} with the data $u$ and $y-y^d$, and $(\bp(y_h),z(y_h))$ denotes the continuous adjoint flux and potential associated with the optimal discrete state $y_h$, i.e., the solution to \eqref{p13b} with right-hand side $y_h-y^d$.
\end{theorem}
\begin{proof}
From the definitions of the variational inequalities and the quasi-interpolation property in \eqref{def:quasi_interpolation}, we have 
\begin{align*}
 \widehat J'(u)(u_h - u) &\ge 0 \qquad \forall\, u_h \in \Uadh = U_h \cap \Uad \subset \Uad, \\
 \widehat J_h'(u_h)(\pi_h u - u_h) &\ge 0 \qquad \forall\, \pi_h u \in \Uadh.
\end{align*}
Adding these inequalities gives
\begin{equation}\label{eqn:f1}
   \underbrace{\big(\widehat J_h'(u_h) - \widehat J'(u)\big)(u - u_h)}_{(1)} + \underbrace{\widehat J_h'(u_h)(\pi_h u - u)}_{(2)} \ge 0.
\end{equation}
We now can estimate (1) in \eqref{eqn:f1} as in the proof of Theorem \ref{thm:nodisc_adjoint_representation}, which gives contributions to the upper part of \eqref{gge}. 
To treat (2) in \eqref{eqn:f1} we write
\begin{align}\label{eqn:f5}
J_h'(u_h)(\pi_h u - u)
  &= \langle\, \omega u_h - \epsilon^{\hF} \, \bp_h \cdot \bn + \kappa_z z_h ,\, \pi_h u - u \,\rangle \nonumber \\[0.3em]
  &= \langle\, \omega (u_h-u),\, \pi_h u - u \,\rangle \nonumber\\
  &\quad
     + \langle\, \epsilon^{\hF} \, (\bp(y_h) -\bp_h)\cdot \bn 
       + \kappa_z (z_h - z(y_h)),\, \pi_h u - u \,\rangle  \nonumber\\
  &\quad
     + \langle\, \epsilon^{\hF} \, (\bp(u)-\bp(y_h))\cdot \bn
       + \kappa_z (z(y_h) - z(u)),\, \pi_h u - u \,\rangle  \nonumber\\
  &\quad
     + \langle\, \omega u -  \epsilon^{\hF} \,\bp(u)\cdot \bn + \kappa_z z(u),\, \pi_h u - u \,\rangle.
\end{align}
Here, $\bp(u),z(u)$ denote the optimal adjoint flux and adjoint potential, respectively, while $\bp(y_h),z(y_h)$ denote auxiliary continuous adjoint flux and potential associated with the optimal discrete adjoint state $y_h$, i.e., with right-hand side $y_h - y^d \in L^2(\Omega)$ in \eqref{p13b}. We already  note at this stage of the proof that the continuous adjoint variable $z$ is zero on the boundary. Straightforward estimation gives
\begin{multline}\label{extra-term}
J_h'(u_h)(\pi_h u - u) \le \frac{\omega}{4}\|u-u_h\|_{0,\Gamma}^2+C(\omega) \|\pi_hu-u\|_{0,\Gamma}^2  \\ +C \big(\epsilon^{1/2}\|(\bp_h - \bp(y_h))\cdot\bn\|_{0,\Gamma}+\kappa_z\|z_h - z(y_h)\|_{0,\Gamma}\big)\|\pi_hu-u\|_{0,\Gamma}  \\ + C\epsilon^{1/2}\|(\bp(y_h)-\bp(u))\cdot\bn\|_{0,\Gamma}\|\pi_hu-u\|_{0,\Gamma}  \\ + C\|(u,\epsilon^{1/2}\bp(u)\cdot\bn)\|_{s,\Gamma}\|\pi_hu-u\|_{-s,\Gamma}.
\end{multline}
Since $y_h, y^d\in L^2(\Omega)$, we have 
$z(y_h) \in H^2(\Omega) \cap H_0^1(\Omega)$. The trace theorem together with the continuity of the continuous adjoint solution operator gives 
\begin{equation}\label{eqn:f7}
\|(\bp(y_h) - \bp(u))\cdot\bn\|_{0,\Gamma}\leq  C \|z(y_h) - z(u)\|_{2,\Omega}  \leq C \|y_h - y \|_{0,\Omega},
\end{equation}
so that
\[C \epsilon^{1/2}\|(\bp(y_h) - \bp(u))\cdot\bn\|_{0,\Gamma}\|\pi_hu-u\|_{0,\Gamma} \le \frac{1}{4}\|y-y_h\|_{0,\Omega}^2 + C\|\pi_hu-u\|_{0,\Gamma}^2.\]
Sorting the addends of \eqref{extra-term} in accordingly gives the estimate \eqref{gge}.
\end{proof}

\subsection{Discussion of Theorem \ref{thm:nodisc_adjoint_representation} and Theorem \ref{thm:fully-discrete}:}\label{discussion} In both cases, namely the variational discrete and the fully discrete setting, to obtain error estimates for the optimal control and the associated state we need LDG error estimates for the state, the adjoint flux as well as for the adjoint potential under natural regularity requirements on the continuous optimal control $u$ and the data $y^d$ (and $f$).

Let us first recall what is known in the literature with respect to the approximation properties of LDG approximations. In the first instance we here use results of \cite{PCastillo_BCockburn_IPergugia_DSchotzau_2000}. In particular, from Theorem~2.1 and Tables~2.1 and~4.1 therein, one obtains for the case $k=1$:
\begin{lemma}\label{smooth-solutions}
Let $\big(v,\bw \big) \in V \times \bW$ be the solutions of \eqref{p14} and let $\big(v_h,\bw_h \big) \in V_h \times \bW_h$ be its LDG approximation, i.e., the solution of the discretized problem  (\ref{DG5}). Assume that $v \in H^{s+2}(\Omega)$ for some $s\ge 0$. Then
\begin{eqnarray}\label{lemmaState1}
  |v-v_h|_{1,\Omega} + \|\bw-\bw_h\|_{0,\Omega} \leq C h \|v\|_{2,\Omega}.
\end{eqnarray}
Moreover, the following $L^2$-error bound holds
\begin{eqnarray}\label{lemmaState2}
  \|v-v_h\|_{0,\Omega}  \leq C h^2\|v\|_{2,\Omega}.
\end{eqnarray}
\end{lemma}
If one inspects the proof of this lemma, one recognizes that higher regularity of $(v,\bw)\in H^{s+2}(\Omega)\times H^{s+2}(\Omega)^2$ for $s\in (0,1)$ in our case $k=1$ does not affect the error estimate.
If we look at the variables of our optimization problem we observe that for $y^d\in L^2(\Omega)$ (and $u\in H^{1/2}(\Gamma)$, which is the minimal regularity we obtain for the control $u$), the adjoint potential satisfies $z\in H^2(\Omega)\cap H^1_0(\Omega)$ with flux $\bp\in H^1(\Omega,\mathbb{R}^2)$. In this case, we conclude that
\[
\kappa_z\|z(u)-z_h(u)\|_{0,\Gamma} \lesssim h^{1/2}, \quad \|\bp(u)\cdot \bn-\bp_h(u) \cdot \bn\|_{0,\Gamma} \lesssim h^{1/2},
\]
since in our LDG method $C_{11} \sim \epsilon h^{-1}$. Moreover, with Lemma~\ref{smooth-solutions} we for higher regularity of the involved variables according to Theorem \ref{thm:higher_reg} may not expect an improvement for the case $k=1$.
But we only have $y\in H^1(\Omega)$ with flux $\bq \in H(\mathrm{div},\Omega)$. With this, Lemma~\ref{smooth-solutions} is not applicable. However, in Lemma 4.3 of the recent paper \cite{JJiang_NJWalkington_YYue_2025} treating stability and convergence of the HDG method in approximating solutions to elliptic PDEs, it among other things is shown that for $s\in [0,1]$
\begin{equation}\label{L2est}
\|y-y_h(u)\|_{0,\Omega} \le Ch^{s+1} |y|_{H^{s+1}(\Omega)}
\end{equation}
holds. \z{Assuming that this estimate can also be established within the LDG framework, one may then expect a corresponding error bound}
\[\| u - u_h \|_{0,\Gamma}  + \|y-y_h\|_{0,\Omega} \lesssim h + h^{1/2} + h^{1/2} \lesssim h^{1/2}.\]
This compares very well with the results of \cite{WGong_WHu_MMateos_JSingler_XZhang_YZhang_2018} for the case $k=0$, that is, piecewise constant flux and piecewise linear potential approximation.

In the fully discrete case, additional terms arise that need to be estimated. For the error associated with the quasi-interpolation operator $\pi_h$, one has for $u\in H^s(\Gamma)$ and  $0 \le s \le 1$ from \eqref{ineq:quasi_interpolationa}-\eqref{ineq:quasi_interpolationb}
\[\|u-\pi_hu\|_{0,\Gamma} \lesssim h^s \quad  \text{ and } \quad \|u-\pi_hu\|_{-s,\Gamma} \lesssim h^{2s}.\]
For the generic case $s=\frac{1}{2}$, this gives the convergence behavior $h^{1/2}$ and $h^1$, respectively. We also need to estimate the approximation errors for $\bp(y_h)$ and $z(y_h)$, where $y_h$ denotes the discrete optimal potential. Since $y_h\in L^2(\Omega)$ with 
\[\|y_h\|_{0,\Omega} \le C\]
uniformly in $h$ (this directly follows from $\|y_h\|_{0,\Omega} \le \sqrt{4\widehat J_h(0) + 2\|y^d\|_{0,\Omega}^2}$) we conclude from Lemma~\ref{smooth-solutions}
\[\|(\bp_h - \bp(y_h))\cdot\bn\|_{0,\Gamma}+\kappa_z\|z_h - z(y_h)\|_{0,\Gamma} \lesssim h^{1/2}.\]

Finally, since $\bp\cdot \bn \in H^{1/2}(\Gamma)$ (see, e.g., \cite{TApel_MMateos_JPfefferer_ARosch_2015}) we may apply \eqref{ineq:quasi_interpolationb} with $s=\frac{1}{2}$. All together, we for $u\in H^{1/2}(\Gamma)$ also in the fully discrete case obtain the error estimate
\[\| u - u_h \|_{0,\Gamma}  + \|y-y_h\|_{0,\Omega} \lesssim h^{1/2}.\]

\begin{remark}\label{Cockburn-extension}
We note that in the work of Cockburn et al. \cite{BCockburn_JGopalakrishnan_FJSayas_2010} an error analysis for a LDG scheme is presented, which fits to our setting of Section \ref{sec:ldg}. From Theorem 2.1 in combination with Theorems 3.1 and 4.1 of this work it is possible to deduce the error estimates
\[
\|y-y_h\|_{0,\Omega}\le C\left\{ h^{l_y+1}|y|_{l_{y}+1,\Omega} + h^{l_q+2} (|\text{div } \bq|_{l_q,\Omega}+|\bq|_{l_q+1, \Omega})\right\}
\]
and
\[
\|\bq-\bq_h\|_{0,\Omega}\le C\left\{ h^{l_y}|y|_{l_{y}+1,\Omega} + h^{l_q+1} |\bq|_{l_q+1,\Omega}\right\},
\]
where $l_y,l_q \in [0,k]$, with $k$ denoting the polynomial degree in our LDG scheme. In the present work, we use $k=1$. In this setting, at least $H^1(\Omega)$-regularity of the fluxes is required and we obtain a more detailed version of the estimate of Lemma \ref{smooth-solutions}. In our present optimal control setting, we only have the generic regularity $\bq\in H(\mathrm{div},\Omega)$ of the optimal flux and $y\in H^1(\Omega)$ for the optimal potential. Utilizing the fact that $\bq \cdot \bn \in H^{-1/2}(\Gamma)$ in the definition of the respective projections we think that it is possible to extend the proof of \cite[Theorem 2.1]{BCockburn_JGopalakrishnan_FJSayas_2010} to our generic regularity setting, yielding the error estimates
\begin{equation}\label{y-detail}
\|y-y_h\|_{0,\Omega}\le C\left\{ h^{l_y+1}|y|_{l_{y}+1,\Omega} + h^{l_q+1} (|\mathrm{div} \, \bq|_{l_q, \Omega}+|\bq|_{l_{q}\cap H(\mathrm{div}),\Omega})\right\}
\end{equation}
and 
\begin{equation}\label{q-detail}
\|\bq-\bq_h\|_{0,\Omega}\le C\left\{ h^{l_y}|y|_{l_{y}+1,\Omega} + h^{l_q} |\bq|_{l_{q}\cap H(\mathrm{div}),\Omega}\right\},
\end{equation}
where again $l_q,l_y\in [0,k]$. For the case $k=1$ and \zz{$u\in H^{1/2}(\Gamma)$} we for $l_y=l_q=0$ then get the estimate
\begin{equation}\label{L2-est-alternativ}
\|y-y_h\|_{0,\Omega}\le C h\left\{|y|_{1,\Omega} +|\bq|_{ H(\mathrm{div}),\Omega}\right\}
\end{equation}
together with a stability estimate for $\bq_h$. It is then possible to replace the estimate \eqref{L2est} in our convergence discussion by that of \eqref{L2-est-alternativ}. Moreover, only $\bq\in H^{l_q}(\Omega)\cap H(div,\Omega)$ is required. However, in the case $k=1$ estimates \eqref{y-detail}-\eqref{q-detail} do not deliver improved error estimates for the errors in optimal control and state, since the regularity of the adjoint variables obtained for $s\in [\frac{1}{2},\frac{3}{2})$ from Theorem \ref{thm:higher_reg} does not contribute to improved convergence order in the case $k=1$. This would only pay off for $k\ge 2$, compare also with the case $k\ge 1$ of \cite{WGong_WHu_MMateos_JSingler_XZhang_YZhang_2018} which would relate to the case $k\ge 2$ of our setting. We leave it to the interested reader to inspect the details.
\end{remark}
\z{To conclude, for both control discretization strategies in the case $k=1$, we only obtain convergence estimates with leading factor $h^{1/2}$, so that the variational discretization concept in fact simplifies the numerical analysis, but does not lead to improved convergence estimates, as it is frequently observed in PDE-constrained optimization \cite[Chapter 3]{MHinze_RPinnau_MUlbrich_SUlbrich_2009a}.} 

%% file: R2-numeric.tex
\section{Numerical Experiments} \label{sec:numeric}


In this section, we provide numerical experiments to demonstrate the performance of the local discontinuous Galerkin discretization and to underpin the theoretical findings established for the Dirichlet boundary control problem. All numerical results are obtained using piecewise linear approximations for the state $y$, the state flux $\bq$, the adjoint $z$, the adjoint flux $\bp$, and the control $u$. \x{Unless otherwise stated, the initial mesh is constructed by partitioning the domain $\Omega$ into a $2 \times 2$ array of uniform squares, each of which is subsequently subdivided into two triangles. The resulting triangulation is then uniformly refined by subdividing each triangle into four congruent sub-triangles; see Fig.~\ref{Fig:regular_mesh}.} To solve the discretized  control constrained problem, the primal-dual active set algorithm is applied as a semismooth Newton method; see, e.g., \cite{MBergounioux_KIto_KKunisch_1999a}. The optimization procedure is terminated when two consecutive active sets coincide. Moreover, the experimental order of convergence is computed using
\[
\hbox{rate} = \frac{1}{\ln 2} \ln \left( \frac{\|e(h)\|_{0,\Omega}}{\|e(h/2)\|_{0,\Omega}}\right),
\]
where $e(h)$ denotes the error associated with a triangulation of mesh size $h$. \x{Although not discussed in Section \ref{discussion}, for the convenience of the reader, in all tables in addition to $\|y-y_h\|_{0,\Omega}$ and $\|u-u_h\|_{0,\Gamma}$ we also report the numerically observed convergence orders for $\|z-z_h\|_{0,\Gamma}$ and $\|(\bp-\bp_h)\cdot \mathbf n\|_{0,\Gamma}$.}

\begin{figure}[H]
\begin{center}
\includegraphics[width=0.8\textwidth]{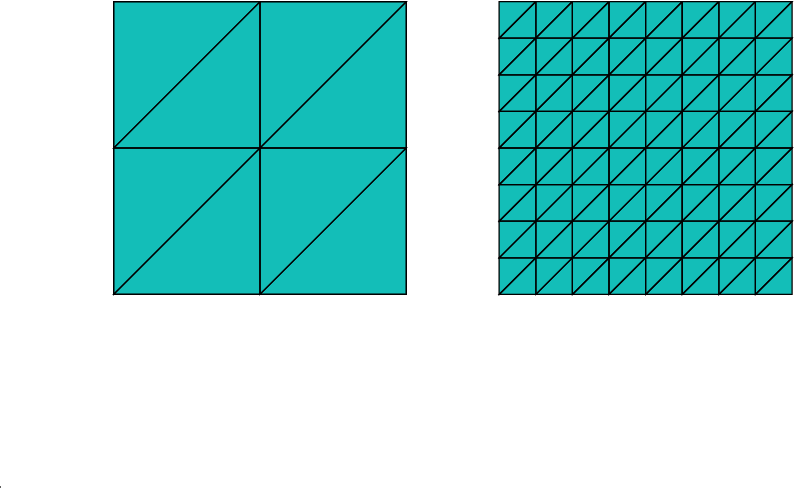}
\vspace{-3cm}
\caption{Initial mesh (left) and  uniformly refined mesh (right).}
\label{Fig:regular_mesh}
\end{center}
\end{figure}

\subsection{Example 1}\label{Ex1}

Our first example is a modified form of the elliptic problem in \cite{SMay_RRannacher_BVexler_2013} and represents an unconstrained problem with analytical solutions. The problem data are chosen as
\[
\Omega=(0,1) \times (0,1),  \quad \beta = (1,1)^T, \quad \alpha=1, \quad \omega=1.
\]
The source function \(f(x_1, x_2)\) and the desired state \(y^d(x_1, x_2)\) are generated so that the analytical solutions for the state \(y\), adjoint \(z\), and control \(u\) are given by
\begin{eqnarray*}
y(x_1,x_2) &=& - \frac{\epsilon^{1/2}}{\omega} \left( x_1 (1-x_1) + x_2 (1-x_2) \right), \\
z(x_1,x_2) &=&  \epsilon^{-1/2}x_1 x_2(1-x_1)(1-x_2), \\
u(x_1,x_2) &=& - \frac{\epsilon^{1/2}}{\omega} \left( x_1 (1-x_1) + x_2 (1-x_2) \right),
\end{eqnarray*}
respectively.

\x{In the present example we have a polynomial exact solution of the optimal control problem. For the LDG approximation errors of these functions, we may expect the best possible order in terms of the mesh size of the finite element mesh. Inspection of the errors on the right-hand side of estimate \eqref{gge} shows that then, the limiting factors in the case $\epsilon=1$ are the terms $\|(\bp-\bp_h(u))\cdot \mathbf n\|_{0,\Gamma}$ and $\kappa_z\|z-z_h(u)\|_{0,\Gamma}$, where $\kappa_z = \frac{-\epsilon^{3/2}}{h}+\chi_{\Gamma^-}|\beta\cdot\bn|$. Numerical tests (not explicitly displayed here) performed for $\epsilon \in [10^{-8},1]$ show that $\|(\bp-\bp_h(u))\cdot \mathbf n\|_{0,\Gamma}$ converges with order one, and $\|z-z_h(u)\|_{0,\Gamma}$ with order two. This behavior is also observed for the approximation of the state $y$ and the flux $\bq$. It is also expected that the projection errors of $u$ converge with best possible order. This altogether for $\epsilon =1$ could explain the convergence order of one for the $L^2(\Gamma)-$error in the optimal control. The $L^2-$error in the optimal state however seems to converge with the higher order of $3/2$. If one compares to the expected rates in the classical finite element approach taken in e.g., \cite{SMay_RRannacher_BVexler_2013}, we observe that our LDG errors in the optimal control with order one and optimal state with order 3/2 behave as the convergence orders proven there. We observe this behaviour of the optimal LDG approximation also in the subsequent examples. With decreasing  $\epsilon$ the convergence behaviour of control and state firstly start to improve and then for further decrease of $\epsilon$ become irregular. This may be explained by the fact that for very small $\epsilon$, one approximates zero functions. In contrast, the convergence behaviour of the adjoint state remains stable.}

\begin{figure}[H]
\begin{center}
\includegraphics[width=1\textwidth]{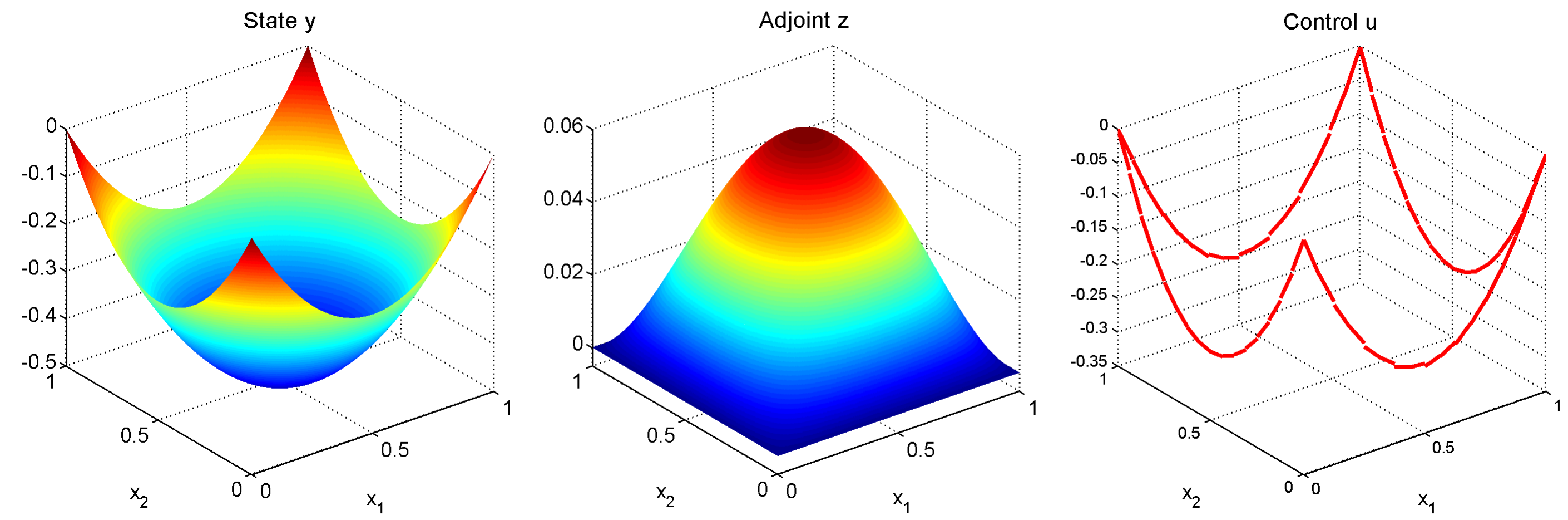}
\caption{Example~\ref{Ex1}: Computed solutions of state $y$, adjoint $z$, and control $u$ (from left to right) with $\epsilon=1$.}
\label{Fig:Ex1_unc}
\end{center}
\end{figure}

\begin{table}[H]
\centering
\caption{Example~\ref{Ex1}: Errors and convergence rates with $\epsilon=1$.}
\resizebox{\textwidth}{!}
{%
\begin{tabular}{c|c | c c | c c | c c | c c}
$h/\sqrt{2}$
&\# elements 
& $\|y-y_h\|_{0,\Omega}$ & rate
& $\|u-u_h\|_{0,\Gamma}$ & rate
& $\|z-z_h\|_{0,\Gamma}$ & rate
& $\|(\bp-\bp_h)\cdot \mathbf n\|_{0,\Gamma}$ & rate \\
\hline
$2^{-1}$ & 32     & 1.64e-02 &  -   & 4.27e-02 &  -   & 2.17e-02 &  -   & 1.03e-01 &  -   \\
$2^{-2}$ & 128    & 3.73e-03 & 2.14 & 2.14e-02 & 1.00 & 5.24e-03 & 2.05 & 5.53e-02 & 0.90 \\
$2^{-3}$ & 512    & 9.97e-04 & 1.90 & 1.12e-02 & 0.94 & 1.28e-03 & 2.03 & 2.91e-02 & 0.93 \\
$2^{-4}$ & 2048   & 3.00e-04 & 1.73 & 5.80e-03 & 0.95 & 3.13e-04 & 2.03 & 1.49e-02 & 0.97 \\
$2^{-5}$ & 8192   & 9.73e-05 & 1.62 & 2.97e-03 & 0.97 & 7.73e-05 & 2.02 & 7.52e-03 & 0.99 \\
$2^{-6}$ & 32768  & 3.29e-05 & 1.56 & 1.50e-03 & 0.98 & 1.92e-05 & 2.01 & 3.78e-03 & 0.99 \\
$2^{-7}$ & 131072 & 1.14e-05 & 1.53 & 7.55e-04 & 0.99 & 4.77e-06 & 2.01 & 1.89e-03 & 1.00 \\
$2^{-8}$ & 524288 & 3.98e-06 & 1.52 & 3.79e-04 & 1.00 & 1.19e-06 & 2.00 & 9.48e-04 & 1.00 \\
\hline
\end{tabular}
}
\label{Ex1_combined_table_eps0}
\end{table}

\begin{table}[H]
\centering
\caption{Example~\ref{Ex1}: Errors and convergence rates with $\epsilon=10^{-6}$.}
\resizebox{\textwidth}{!}
{%
\begin{tabular}{c | c | c c | c c | c c | c c}
$h/\sqrt{2}$
&\# elements
& $\|y-y_h\|_{0,\Omega}$ & rate
& $\|u-u_h\|_{0,\Gamma}$ & rate
& $\|z-z_h\|_{0,\Gamma}$ & rate
& $\|(\bp-\bp_h)\cdot \mathbf n\|_{0,\Gamma}$ & rate \\
\hline
$2^{-1}$
&32
& 3.87e-02 & -
& 7.39e-02 & -
& 6.63e+00 & -
& 8.43e-02 & - \\

$2^{-2}$
&128
& 5.33e-03 & 2.86
& 1.00e-02 & 2.88
& 1.80e+00 & 1.88
& 4.54e-02 & 0.89 \\

$2^{-3}$
&512
& 8.50e-04 & 2.65
& 1.54e-03 & 2.71
& 4.70e-01 & 1.94
& 2.35e-02 & 0.95 \\

$2^{-4}$
&2048
& 1.67e-04 & 2.35
& 2.89e-04 & 2.41
& 1.20e-01 & 1.97
& 1.20e-02 & 0.97 \\

$2^{-5}$
&8192
& 4.48e-05 & 1.90
& 7.86e-05 & 1.88
& 3.03e-02 & 1.99
& 6.05e-03 & 0.99 \\

$2^{-6}$
&32768
& 1.99e-05 & 1.17
& 3.99e-05 & 0.98
& 7.62e-03 & 1.99
& 3.04e-03 & 0.99 \\

$2^{-7}$
&131072
& 1.22e-05 & 0.70
& 3.54e-05 & 0.17
& 1.91e-03 & 2.00
& 1.52e-03 & 1.00 \\

$2^{-8}$
&524288
& 4.41e-06 & 1.47
& 2.30e-05 & 0.63
& 4.78e-04 & 2.00
& 7.67e-04 & 0.99 \\

\hline
\end{tabular}
}
\label{Ex1_combined_table_eps6}
\end{table}

Fig.~\ref{Fig:Ex1_unc} shows the computed solutions of the state $y$, adjoint $z$, and control $u$ on a fine mesh with 524288 elements. Tab.~\ref{Ex1_combined_table_eps0} and Tab.~\ref{Ex1_combined_table_eps6} present the $L^2(\Omega)$ errors of the state variable $y$, as well as the $L^2(\Gamma)$ errors of the control $u$, the adjoint state $z$, and the adjoint flux $\bp$, corresponding to $\epsilon = 1$ and $\epsilon = 10^{-6}$, respectively. 


\subsection{Example 2}\label{Ex_singular}

The following example, posed on the unit square $\Omega=(0,1) \times (0,1)$, is adopted from \cite{ECasas_JPRaymond_2006b,WGong_NYan_2011b} and does not admit an explicit analytical solution.  The remaining problem data are specified as
\[
f=0, \quad y^d = \frac{1}{(x_1^2 + x_2^2)^{1/3}}, \quad \beta=(1,1)^T, \quad \alpha=1, \quad \omega=1.
\]
The admissible control set is defined by
\[
\Uad = \bigl\{ u \in L^2(\Gamma) : 0 \leq u(x) \leq 0.2 \;\; \text{a.e. } x \in \Gamma \bigr\}.
\]
Here, the largest interior angle is $\theta=\frac{\pi}{2}$ and $y^d \in H^{1/3 - \eta}(\Omega)$ for any $\eta >0$ due to the singularity on the boundary. Consequently, from Theorem~\ref{thm:higher_reg}, it follows that $s \in [\frac{1}{2}, \frac{5}{6})$.

\begin{figure}[H]
\begin{center}
\includegraphics[width=1\textwidth]{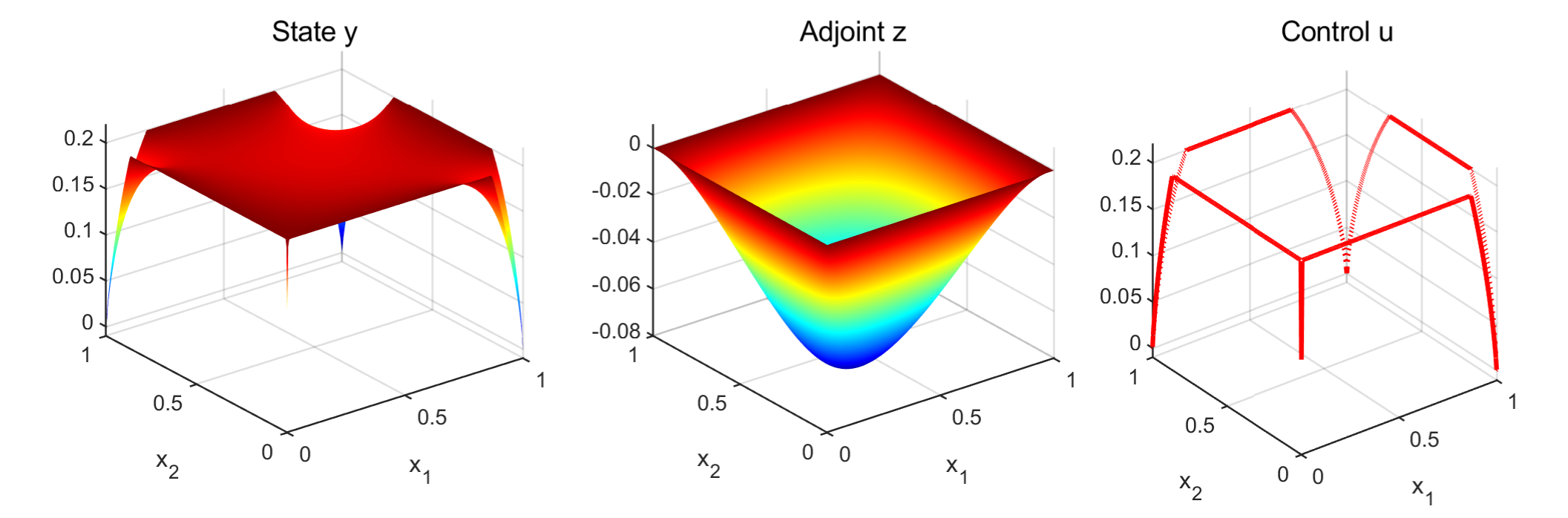}
\caption{Example~\ref{Ex_singular}: Computed solutions of state $y$, adjoint $z$, and control (from left to right) with $\epsilon=1$.}
\label{Fig:Ex_singular_eps0}
\end{center}
\end{figure}

\begin{table}[H]
\centering
\caption{Example~\ref{Ex_singular}: Errors and convergence rates with $\epsilon=1$.}
\resizebox{\textwidth}{!}
{%
\begin{tabular}{c |c | c c | c c | c c | c c}
$h/\sqrt{2}$
&\# elements 
& $\|y-y_h\|_{0,\Omega}$ & rate
& $\|u-u_h\|_{0,\Gamma}$ & rate
& $\|z-z_h\|_{0,\Gamma}$ & rate
& $\|(\bp-\bp_h)\cdot \mathbf n\|_{0,\Gamma}$ & rate \\
\hline
$2^{-1}$
&32     
& 3.09e-02 & -    
& 1.14e-01 & -    
& 2.41e-02 & -    
& 1.79e-01 & -    \\

$2^{-2}$
&128    
& 1.35e-02 & 1.20 
& 6.59e-02 & 0.79 
& 6.97e-03 & 1.79 
& 1.12e-01 & 0.68 \\

$2^{-3}$
&512    
& 5.81e-03 & 1.21 
& 3.38e-02 & 0.96 
& 1.86e-03 & 1.90 
& 6.53e-02 & 0.78 \\

$2^{-4}$
&2048   
& 2.53e-03 & 1.20 
& 1.74e-02 & 0.96 
& 4.89e-04 & 1.93 
& 3.66e-02 & 0.84 \\

$2^{-5}$
&8192   
& 1.21e-03 & 1.07 
& 1.00e-02 & 0.80 
& 1.28e-04 & 1.94 
& 1.97e-02 & 0.89 \\

$2^{-6}$
&32768  
& 5.91e-04 & 1.03 
& 6.54e-03 & 0.61 
& 3.24e-05 & 1.98 
& 1.01e-02 & 0.96 \\

$2^{-7}$
&131072 
& 2.62e-04 & 1.17 
& 3.74e-03 & 0.81 
& 6.95e-06 & 2.22 
& 4.62e-03 & 1.13 \\

\hline
\end{tabular}
}
\label{Ex2_combined_table_eps0}
\end{table}

\begin{table}[H]
\centering
\caption{Example~\ref{Ex_singular}: Errors and convergence rates with $\epsilon=10^{-4}$.}
\resizebox{\textwidth}{!}
{%
\begin{tabular}{c |c | c c | c c | c c | c c}
$h/\sqrt{2}$
&\# elements 
& $\|y-y_h\|_{0,\Omega}$ & rate
& $\|u-u_h\|_{0,\Gamma}$ & rate
& $\|z-z_h\|_{0,\Gamma}$ & rate
& $\|(\bp-\bp_h)\cdot \mathbf n\|_{0,\Gamma}$ & rate \\
\hline
$2^{-1}$
&32     
& 1.83e-02 & -    
& 6.04e-02 & -    
& 3.21e-01 & -    
& 2.08e+01 & -    \\

$2^{-2}$
&128    
& 8.93e-03 & 1.03 
& 3.20e-02 & 0.92 
& 2.63e-01 & 0.29 
& 2.06e+01 & 0.01 \\

$2^{-3}$
&512    
& 4.49e-03 & 0.99 
& 1.77e-02 & 0.85 
& 2.31e-01 & 0.18 
& 2.02e+01 & 0.03 \\

$2^{-4}$
&2048   
& 2.44e-03 & 0.88 
& 1.07e-02 & 0.73 
& 2.10e-01 & 0.14 
& 1.94e+01 & 0.06 \\

$2^{-5}$
&8192   
& 1.58e-03 & 0.63 
& 6.85e-03 & 0.64 
& 1.88e-01 & 0.16 
& 1.78e+01 & 0.12 \\

$2^{-6}$
&32768  
& 1.19e-03 & 0.41 
& 4.25e-03 & 0.69 
& 1.54e-01 & 0.29 
& 1.49e+01 & 0.26 \\

$2^{-7}$
&131072 
& 8.37e-04 & 0.51 
& 1.92e-03 & 1.15 
& 9.60e-02 & 0.68 
& 9.45e+00 & 0.66 \\

\hline
\end{tabular}
}
\label{Ex2_combined_table_eps4}
\end{table}

The problem is solved numerically on a fine mesh consisting of 524288 elements (that is, $h=2^{-8}/\sqrt{2}$ and 1572864 degrees of freedom), and the resulting solution is used as a reference for comparison with solutions computed on coarser meshes. Fig.~\ref{Fig:Ex_singular_eps0} displays the numerical solutions computed on the reference mesh for $\epsilon=1$. \x{The numerical results reported in Tab.~\ref{Ex2_combined_table_eps0} for $\epsilon=1$
again show higher convergence rates for $\|u-u_h\|_{0,\Gamma}$ and $\|y-y_h\|_{0,\Omega}$. Let us note that similar observations are also reported for a related example in \cite[Tab.4]{WGong_WHu_MMateos_JSingler_XZhang_YZhang_2018}, which corresponds to our numerical setting. As for the previous example, we observe that the convergence rates are in good agreement with those proved for the classical finite element approach; compare e.g., \cite[Example~3]{SMay_RRannacher_BVexler_2013}. However, from the numerical results reported in Tab.~\ref{Ex2_combined_table_eps4} for $\epsilon=10^{-4}$ it is difficult to deduce an error behavior from the estimate \eqref{gge}, since we expect that the exact state and adjoint develop boundary layers, so that besides the explicit appearance of $\epsilon$ in this estimate one also has to take into account the growth of the respective $H^2$-seminorms of $y$ and $z$ in the estimation of $\|y-y_h(u)\|_{0,\Omega}$ and $\|z-z_h(u)\|_{0,\Gamma}$, which in the case of boundary layers grow critically with decreasing $\epsilon$.}


\subsection{Example 3}\label{Ex_polygonal}

Our last example, adapted from \cite[Example~2]{SMay_RRannacher_BVexler_2013}, is formulated on a polygonal domain with maximum interior angle $\theta = \frac{5}{6} \pi$, as depicted in Fig.~\ref{Fig:Ex_polygonal}.  The remaining problem data are
\[
y^d = \begin{cases}
-1, & 0 \leq x_2 < 0.5, \\
1, & 0.5 \leq x_2 < 1,
\end{cases}
\quad
f = 1, \quad \beta = (1,0)^T, \quad \alpha = 2, \quad \omega = 1.
\]
The admissible control set is defined by
\[
\Uad = \bigl\{ u \in L^2(\Gamma) : 0 \leq u(x) \;\; \text{a.e. } x \in \Gamma \bigr\}.
\]
In this example, $y^d \in H^{1/2-\eta}(\Omega)$ for any $\eta>0$, and the maximum interior angle is $\theta = \frac{5\pi}{6}$. Consequently, it follows that $s \in [\frac{1}{2},\frac{7}{10})$ from Theorem~\ref{thm:higher_reg}.

\begin{figure}[htp!]
  \centering
  \begin{subfigure}[b]{0.30\textwidth}
    \centering
    \begin{tikzpicture}[scale=0.55]
      \draw[->] (2,2) -- (10,2) node[below] {\footnotesize $x_1$};
      \draw[->] (6,2) -- (6,6)  node[left]  {\footnotesize $x_2$};
      \draw (9,2) -- (9,5);
      \draw (9,5) -- (3,5);
      \draw (6,2) -- (3,5);
      \draw[->] (6.3,2) arc (00:150:.24cm);
      \node[] at (9,1.8)  {\footnotesize $1$};
      \node[] at (5.8,5.2){\footnotesize $1$};
      \node[] at (8,4)    {$\Omega$};
      \node[] at (6.2,2.4){$\theta$};
      \draw[dashed] (3,5) -- (3,2);
      \node[] at (3,1.8)  {\footnotesize $-\sqrt{3}$};
    \end{tikzpicture}
    \caption{Domain with $\theta=\frac{5}{6}\pi$.}
    \label{Fig:ex3_domain}
  \end{subfigure}
  \hfill
  \begin{subfigure}[b]{0.65\textwidth}
    \centering
    \includegraphics[width=\textwidth]{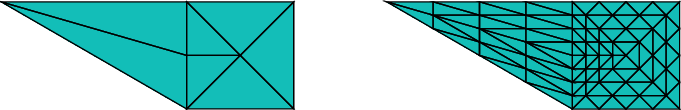}
    \caption{Initial mesh (left) and uniformly refined mesh (right).}
    \label{Fig:ex3_mesh}
  \end{subfigure}
  \caption{Example~\ref{Ex_polygonal}: Domain and meshes.}
  \label{Fig:Ex_polygonal}
\end{figure}

\begin{figure}[H]
\centering
\includegraphics[width=1.10\textwidth,trim=80 80 0 20,clip]{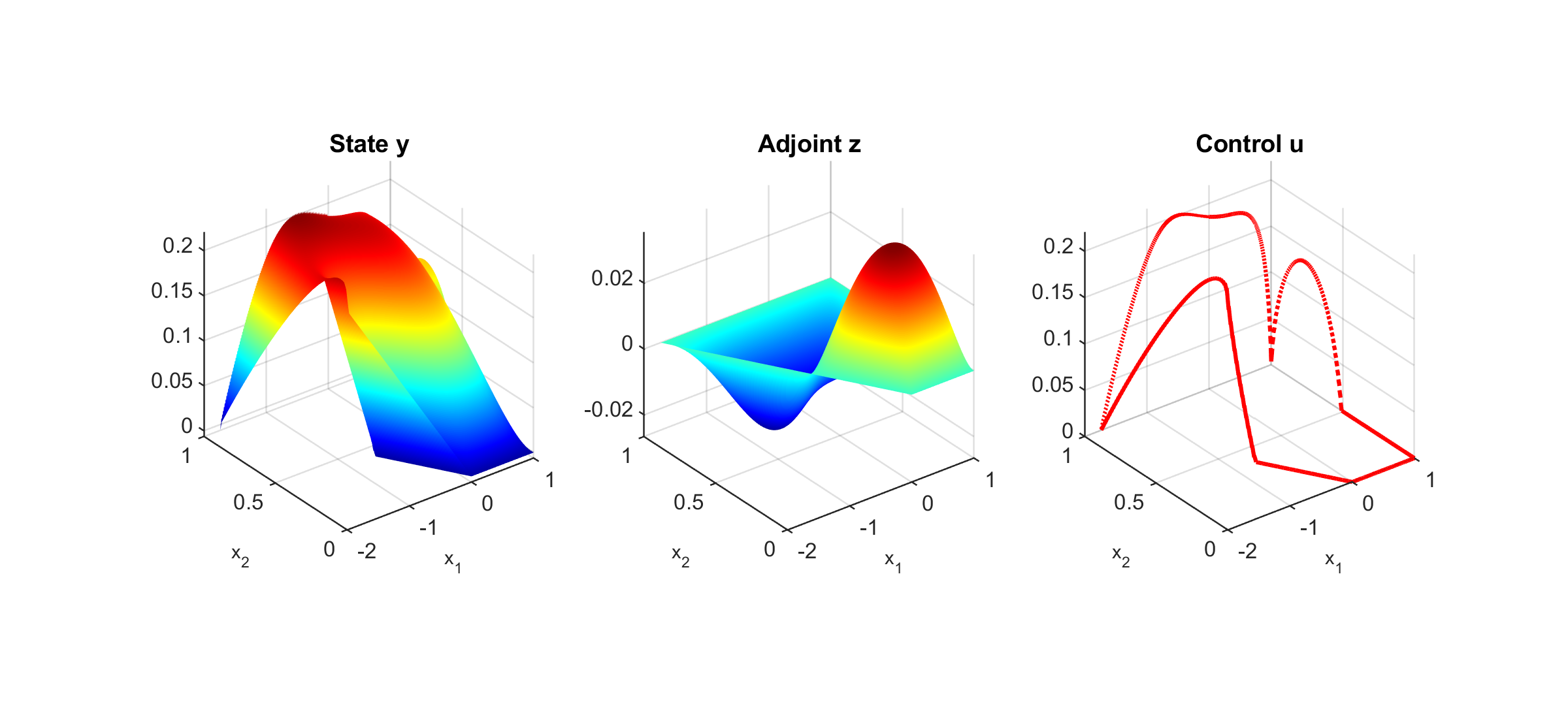}
\caption{Example~\ref{Ex_polygonal}: Computed solutions for state $y$,  adjoint $z$, and control (from left to right) with $\epsilon=1$.}
\label{Fig:Ex_polygonal_eps1}
\end{figure}

\begin{table}[H]
\centering
\caption{Example~\ref{Ex_polygonal}: Errors and convergence rates with $\epsilon=1$.}
\resizebox{\textwidth}{!}
{%
\begin{tabular}{c |c | c c | c c | c c | c c}
$h$
&\# elements 
& $\|y-y_h\|_{0,\Omega}$ & rate
& $\|u-u_h\|_{0,\Gamma}$ & rate
& $\|z-z_h\|_{0,\Gamma}$ & rate
& $\|(\bp-\bp_h)\cdot \mathbf n\|_{0,\Gamma}$ & rate \\
\hline
$2^{0}$
&28     
& 4.75e-02 & -    
& 1.38e-01 & -    
& 1.80e-02 & -    
& 2.27e-01 & -    \\

$2^{-1}$
&112    
& 2.78e-02 & 0.77 
& 9.35e-02 & 0.56 
& 5.02e-03 & 1.85 
& 1.39e-01 & 0.71 \\

$2^{-2}$
&448    
& 1.34e-02 & 1.06 
& 5.36e-02 & 0.80 
& 1.45e-03 & 1.79 
& 7.88e-02 & 0.82 \\

$2^{-3}$
&1792   
& 6.64e-03 & 1.01 
& 2.78e-02 & 0.95 
& 4.11e-04 & 1.82 
& 4.28e-02 & 0.88 \\

$2^{-4}$
&7168   
& 3.27e-03 & 1.02 
& 1.37e-02 & 1.02 
& 1.12e-04 & 1.88 
& 2.27e-02 & 0.91 \\

$2^{-5}$
&28672  
& 1.57e-03 & 1.06 
& 6.37e-03 & 1.10 
& 2.87e-05 & 1.96 
& 1.16e-02 & 0.97 \\

$2^{-6}$
&114688 
& 6.95e-04 & 1.18 
& 2.62e-03 & 1.28 
& 6.16e-06 & 2.22 
& 5.19e-03 & 1.16 \\

\hline
\end{tabular}
}
\label{Ex_polygonal_combined_table_eps0}
\end{table}

\begin{table}[H]
\centering
\caption{Example~\ref{Ex_polygonal}: Errors and convergence rates with $\epsilon=10^{-4}$.}
\resizebox{\textwidth}{!}
{%
\begin{tabular}{c |c | c c | c c | c c | c c}
$h$
&\# elements 
& $\|y-y_h\|_{0,\Omega}$ & rate
& $\|u-u_h\|_{0,\Gamma}$ & rate
& $\|z-z_h\|_{0,\Gamma}$ & rate
& $\|(\bp-\bp_h)\cdot \mathbf n\|_{0,\Gamma}$ & rate \\
\hline
$2^{0}$
&28     
& 8.28e-02 & -    
& 4.16e-02 & -    
& 6.08e-01 & -    
& 1.50e+01 & -    \\

$2^{-1}$
&112    
& 5.82e-02 & 0.51 
& 4.17e-02 & -0.00 
& 6.12e-01 & -0.01 
& 1.48e+01 & 0.02 \\

$2^{-2}$
&448    
& 4.58e-02 & 0.35 
& 3.75e-02 & 0.15 
& 5.82e-01 & 0.07 
& 1.45e+01 & 0.04 \\

$2^{-3}$
&1792   
& 3.74e-02 & 0.29 
& 3.12e-02 & 0.26 
& 4.80e-01 & 0.28 
& 1.38e+01 & 0.07 \\

$2^{-4}$
&7168   
& 2.79e-02 & 0.42 
& 2.36e-02 & 0.41 
& 3.22e-01 & 0.58 
& 1.25e+01 & 0.15 \\

$2^{-5}$
&28672  
& 1.68e-02 & 0.74 
& 1.45e-02 & 0.70 
& 2.13e-01 & 0.59 
& 1.00e+01 & 0.31 \\

$2^{-6}$
&114688 
& 8.19e-03 & 1.03 
& 6.49e-03 & 1.16 
& 1.20e-01 & 0.82 
& 5.95e+00 & 0.75 \\

\hline
\end{tabular}
}
\label{Ex_polygonal_combined_table_eps4}
\end{table}

The reference solutions have been computed on a fine mesh with 458752 elements (that is, 1376256 degrees of freedom); see Fig.~\ref{Fig:Ex_polygonal_eps1} for the case $\epsilon=1$. \x{In this case, we again observe better convergence rates than predicted by our theory. For the state we observe a convergence rate close to 1.1, which is the rate proven in \cite{SMay_RRannacher_BVexler_2013} for the classical finite element approach. According to this reference, the control error for the classical approach should converge with order 0.6. Here, however, we observe an order between 1 and 1.2. Concerning the behaviour for small $\epsilon$, a discussion similar to that of Example~\ref{Ex_singular} applies. The numerical results are reported in Tab.~\ref{Ex_polygonal_combined_table_eps0} for $\epsilon=1$ and Tab.~\ref{Ex_polygonal_combined_table_eps4} for $\epsilon=10^{-4}$.}

%% file: R2-conclusion.tex
\section{Conclusions}\label{sec:conclusion} 

In this work, we have investigated Dirichlet boundary control of a convection--diffusion equation with $L^2$-boundary controls subject to pointwise constraints, discretized using the LDG method. The LDG framework naturally incorporates the Dirichlet boundary conditions into the variational formulation, even when the control space is chosen as $L^2(\Gamma)$. We have derived \x{general a priori} error estimates for both the fully discrete and the variational discrete approximation of the optimal control problem in convex polygonal domains and presented numerical results that for moderate sizes of $\epsilon$ underpin the theoretical predictions \x{for the numerical approximation of Dirichlet boundary control problems}. As future work, the development of a posteriori error estimates and adaptive discontinuous Galerkin methods could help to further robustify the numerical treatment in the convection-dominated case.

%% file: dirichlet_boundary_control_LDG.bbl
\begin{thebibliography}{10}

\bibitem{MHinze_2005a}
M.~Hinze, ``A variational discretization concept in control constrained
  optimization: the linear-quadratic case,'' {\em Comput. Optim. Appl.},
  vol.~30, pp.~45--61, 2005.

\bibitem{MHinze_RPinnau_MUlbrich_SUlbrich_2009a}
M.~Hinze, R.~Pinnau, M.~Ulbrich, and S.~Ulbrich, {\em Optimization with PDE Constraints}, vol.~23 of {\em Mathematical Modelling: Theory and
  Applications}.
\newblock Springer, 2009.

\bibitem{WGong_NYan_2011b}
W.~Gong and N.~Yan, ``Mixed finite element method for {D}irichlet boundary
  control problem governed by elliptic {PDE}s,'' {\em SIAM J. Control Optim.},
  vol.~49, no.~3, pp.~984--1014, 2011.

\bibitem{JPfefferer_BVexler_2025}
J.~Pfefferer and B.~Vexler, ``Numerical analysis for {D}irichlet optimal
  control problems on convex polyhedral domains,'' {\em Numer. Math.},
  vol.~157, pp.~1937--1974, 2025.

\bibitem{WGong_WHu_MMateos_JSingler_XZhang_YZhang_2018}
W.~Gong, W.~Hu, M.~Mateos, J.~Singler, X.~Zhang, and Y.~Zhang, ``A new {HDG}
  method for {D}irichlet boundary control of convection diffusion {PDEs} {II}:
  Low regularity,'' {\em SIAM J. Numer. Anal.}, vol.~56, no.~4, pp.~2262--2287,
  2018.

\bibitem{FTroeltzsch_2010a}
F.~Tr{\"o}ltzsch, {\em Optimal Control of Partial Differential Equations:
  {T}heory, Methods and Applications}, vol.~112 of {\em Graduate Studies in
  Mathematics}.
\newblock Providence, RI: American Mathematical Society, 2010.

\bibitem{BVexler_DMeidner_2025}
B.~Vexler and D.~Meidner, {\em Numerical Analysis for Elliptic Optimal Control
  Problems}, vol.~67 of {\em Springer Series in Computational Mathematics}.
\newblock Springer Cham, 2025.

\bibitem{NArada_ECasas_FTroeltzsch_2002a}
N.~Arada, E.~Casas, and F.~Tr{\"o}ltzsch, ``Error estimates for the numerical
  approximation of a semilinear elliptic control problem,'' {\em Comput. Optim.
  Appl.}, vol.~23, no.~2, pp.~201--229, 2002.

\bibitem{ECasas_MMateos_2008a}
E.~Casas and M.~Mateos, ``Error estimates for the numerical approximation of
  {N}eumann control problems,'' {\em Comput. Optim. Appl.}, vol.~39, no.~3,
  pp.~265--295, 2008.

\bibitem{TGeveci_1979}
T.~Geveci, ``On the approximation of the solution of an optimal control problem
  governed by an elliptic equation,'' {\em RAIRO Anal. Numer.}, vol.~13, no.~4,
  pp.~313--328, 1979.

\bibitem{MHinze_UMatthes_2009a}
M.~Hinze and U.~Matthes, ``A note on variational discretization of elliptic
  {N}eumann boundary control,'' {\em Control Cybern.}, vol.~38, no.~3,
  pp.~577--591, 2009.

\bibitem{NArada_JPRaymond_2002a}
N.~Arada and J.-P. Raymond, ``Dirichlet boundary control of semilinear
  parabolic equations. {I}. {P}roblems with no state constraints,'' {\em Appl.
  Math. Optim.}, vol.~45, no.~2, pp.~125--143, 2002.

\bibitem{FBBelgacem_HEFekih_HMetoui_2003}
F.~B. Belgacem, H.~E. Fekih, and H.~Metoui, ``Singular perturbation for the
  {D}irichlet boundary control of elliptic problems,'' {\em M2AN Math. Model.
  Numer. Anal.}, vol.~37, no.~5, pp.~883--850, 2003.

\bibitem{TApel_MMateos_JPfefferer_ARosch_2018}
T.~Apel, M.~Mateos, J.~Pfefferer, and A.~R\"osch, ``Error estimates for
  {D}irichlet control problems in polygonal domains: quasi--uniform meshes,''
  {\em Math. Control and Rel. Fields}, vol.~8, no.~1, pp.~217--245, 2018.

\bibitem{ECasas_JPRaymond_2006b}
E.~Casas and J.-P. Raymond, ``Error estimates for the numerical approximation
  of {D}irichlet boundary control for semilinear elliptic equations,'' {\em
  SIAM J. Control Optim.}, vol.~45, no.~5, pp.~1586--1611, 2006.

\bibitem{KDeckelnick_AGunther_MHinze_2009}
K.~Deckelnick, A.~G\"unther, and M.~Hinze, ``Finite element approximation of
  {D}irichlet boundary control for elliptic {PDEs} on two--and
  three--dimensional curved domains,'' {\em SIAM J. Control Optim.}, vol.~48, no.~4, pp.~2798--2819, 2009.

\bibitem{SMay_RRannacher_BVexler_2013}
S.~May, R.~Rannacher, and B.~Vexler, ``Error analysis for a finite element
  approximation of elliptic {D}irichlet boundary control problems,'' {\em SIAM
  J. Control Optim.}, vol.~51, no.~3, pp.~2585--2611, 2013.

\bibitem{MMateos_2018}
M.~Mateos, ``Optimization methods for {D}irichlet control problems,'' {\em
  Optimization}, vol.~67, no.~5, pp.~585--617, 2018.

\bibitem{TApel_MMateos_JPfefferer_ARosch_2015}
T.~Apel, M.~Mateos, J.~Pfefferer, and A.~R\"osch, ``On the regularity of the
  solutions of {D}irichlet optimal control problems in polygonal domains,''
  {\em SIAM J. Control Optim.}, vol.~53, no.~6, pp.~3620--3641, 2015.

\bibitem{GOf_TXPhan_OSteinbach_2010}
G.~Of, T.~X. Phan, and O.~Steinbach, ``Boundary element methods for {D}irichlet
  boundary control problems,'' {\em Math. Method Appl. Sci.}, vol.~33, no.~18,
  pp.~2187--2205, 2010.

\bibitem{MWinkler_2020}
M.~Winkler, ``Error estimates for variational normal derivatives and
  {D}irichlet control problems with energy regularization,'' {\em Numer.
  Math.}, vol.~144, pp.~413--445, 2020.

\bibitem{SChowdhury_TGudi_AKNandakumaran_2017}
S.~Chowdhury, T.~Gudi, and A.~K. Nandakumaran, ``Error bounds for a {D}irichlet
  boundary control problem based on energy spaces,'' {\em Math. Comput.},
  vol.~86, no.~305, pp.~1103--1126, 2017.

\bibitem{MKarkulik_2020}
M.~Karkulik, ``A finite element method for elliptic {D}irichlet boundary
  control problems,'' {\em Comput. Methods Appl. Math.}, vol.~20, no.~4,
  pp.~827--843, 2020.

\bibitem{SDu_XHe_2023}
S.~Du and X.~He, ``Finite element approximation to optimal {D}irichlet boundary
  control problem: A priori and a posteriori error estimates,'' {\em Comput.
  Math. Appl.}, vol.~131, pp.~14--25, 2023.

\bibitem{WWHu_JGShen_JRSingler_YWZhang_XBZheng_2020}
W.~W. Hu, J.~Shen, J.~R. Singler, Y.~Zhang, and X.~B. Zheng, ``A
  superconvergent hybridizable discontinuous {G}alerkin method for {D}irichlet
  boundary control of elliptic {PDEs},'' {\em Numer. Math.}, vol.~144,
  pp.~375--411, 2020.

\bibitem{WGong_WHu_MMateos_JRSingler_YZhang_2020}
W.~Gong, W.~Hu, M.~Mateos, J.~R. Singler, and Y.~Zhang, ``Analysis of a
  hybridizable discontinuous {G}alerkin scheme for the tangential control of
  the {S}tokes system,'' {\em ESAIM: M2AN}, vol.~54, no.~6, pp.~2229--2264,
  2020.

\bibitem{WGong_MMateos_JSingler_YZhang_2022}
W.~Gong, M.~Mateos, J.~Singler, and Y.~Zhang, ``Analysis and approximations of
  {D}irichlet boundary control of {S}tokes flows in the energy space,'' {\em
  SIAM J. Numer. Anal.}, vol.~60, no.~1, pp.~450--474, 2022.

\bibitem{AVFursikov_MDGunzburger_LSHou_1998a}
A.~V. Fursikov, M.~D. Gunzburger, and L.~S. Hou, ``Boundary value problems and
  optimal boundary control for the {N}avier--{S}tokes systems: The
  two--dimensional case,'' {\em SIAM J. Control Optim.}, vol.~36, no.~3, pp.~852--894,
  1998.

\bibitem{MHinze_KKunisch_2004}
M.~Hinze and K.~Kunisch, ``Second order methods for boundary control of the
  instationary {N}avier--{S}tokes system,'' {\em ZAMM Z. Angew. Math. Mech.},
  vol.~84, pp.~171--187, 2004.

\bibitem{WHu_MMateos_JSingler_YZhang_2018}
W.~Hu, M.~Mateos, J.~Singler, and Y.~Zhang, ``A new {HDG} method for
  {D}irichlet boundary control of convection diffusion {PDEs} {I}: High
  regularity,'' tech. rep., 2018.
\newblock arXiv:1801.01461v1.

\bibitem{HChen_JRSingler_YZhang_2019}
H.~Chen, J.~R. Singler, and Y.~Zhang, ``An {HDG} method for {Dirichlet}
  boundary control of convection dominated diffusion {PDEs},'' {\em SIAM J.
  Numer. Anal.}, vol.~57, no.~4, pp.~1919--1946, 2019.

\bibitem{CCorekli_2022}
C.~Corekli, ``The {SIPG} method of {Dirichlet} boundary optimal control
  problems with weakly imposed boundary conditions,'' {\em AIMS Math.}, vol.~7,
  no.~4, pp.~6711--6742, 2022.

\bibitem{BCockburn_JGopalakrishnan_RLazarov_2009}
B.~Cockburn, J.~Gopalakrishnan, and R.~Lazarov, ``Unified hybridization of discontinuous
{G}alerkin, mixed, and continuous Galerkin methods for second order elliptic problems,'' {\em SIAM J. Numer. Anal.}, vol.~47,
  no.~2, pp.~1319--1365, 2009.
  

\bibitem{ABuffa_TJRHughes_GSangalli_2006a}
A.~Buffa, T.~J.~R. Hughes, and G.~Sangalli, ``Analysis of a multiscale
  discontinuous {G}alerkin method for convection-diffusion problems,'' {\em
  SIAM J. Numer. Anal.}, vol.~44, no.~4, pp.~1420--1440, 2006.

\bibitem{PCastillo_BCockburn_DSchotzau_CSchwab_2002}
P.~Castillo, B.~Cockburn, D.~Sch{\"o}tzau, and C.~Schwab, ``Optimal a priori
  error estimates for the hp-version of the local discontinuous {G}alerkin
  method for convection-diffusion problems,'' {\em Math. Comp.}, vol.~71, no.~238,
  pp.~455--478, 2002.

\bibitem{JCesenek_MFeistauer_2012}
J.~\v{C}esenek and M.~Feistauer, ``Theory of the space-time discontinuous
  {G}alerkin method for nonstationary parabolic problems with nonlinear
  convection and diffusion,'' {\em SIAM J. Numer. Anal.}, vol.~50, no.~3,
  pp.~1181--1206, 2012.

\bibitem{YCheng_CWShu_2010}
Y.~Cheng and C.-W. Shu, ``Superconvergence of discontinuous {G}alerkin and
  local discontinuous {G}alerkin schemes for linear hyperbolic and
  convection-diffusion equations in one space dimension,'' {\em SIAM J. Numer.
  Anal.}, vol.~47, no.~6, pp.~4044--4072, 2010.

\bibitem{BCockburn_CWShu_1998b}
B.~Cockburn and C.-W. Shu, ``The local discontinuous {G}alerkin method for
  time-dependent convection-diffusion systems,'' {\em SIAM J. Numer. Anal.},
  vol.~35, no.~6, pp.~2440--2463, 1998.

\bibitem{PCastillo_BCockburn_IPergugia_DSchotzau_2000}
P.~Castillo, B.~Cockburn, I.~Perugia, and D.~Sch{\"o}tzau, ``An a priori error
  analysis of the local discontinuous {G}alerkin method for elliptic
  problems,'' {\em SIAM J. Numer. Anal.}, vol.~38, no.~5, pp.~1676--1706, 2000.

\bibitem{BCockburn_JGopalakrishnan_FJSayas_2010}
B.~Cockburn, J.~Gopalakrishnan, and F.~J. Sayas, ``A projection-based error
  analysis of {HDG} methods,'' {\em Math. Comp.}, vol.~79, no.~271,
  pp.~1351--1367, 2010.

\bibitem{BCockburn_GKanschat_DSchotzau_2004a}
B.~Cockburn, G.~Kanschat, and D.~Sch{\"o}tzau, ``The local discontinuous {G}alerkin method for the {O}seen equations,'' {\em Math. Comp.}, vol.~73, no.~246,  pp.~569--593, 2004.

\bibitem{JJiang_NJWalkington_YYue_2025}
J.~Jiang, N.~J. Walkington, and Y.~Yue, ``Stability and convergence of {HDG}
  schemes under minimal regularity,'' {\em SIAM J. Numer. Anal.}, vol.~63,
  no.~5, pp.~2048--2071, 2025.

\bibitem{PBenner_HYucel_2017}
P.~Benner and H.~Y\"ucel, ``Adaptive symmetric interior penalty {G}alerkin
  method for boundary control problems,'' {\em SIAM J. Numer. Anal.}, vol.~55,
  no.~2, pp.~1101--1133, 2017.

\bibitem{DLeykekhman_MHeinkenschloss_2012a}
D.~Leykekhman and M.~Heinkenschloss, ``Local error analysis of discontinuous
  {G}alerkin methods for advection-dominated elliptic linear-quadratic optimal
  control problems,'' {\em SIAM J. Numer. Anal.}, vol.~50, no.~4,
  pp.~2012--2038, 2012.

\bibitem{HYucel_PBenner_2015a}
H.~Y\"ucel and P.~Benner, ``Adaptive discontinuous {G}alerkin methods for state
  constrained optimal control problems governed by convection diffusion
  equations,'' {\em Comput. Optim. Appl.}, vol.~62, pp.~291--321, 2015.

\bibitem{HYucel_MHeinkenschloss_BKarasozen_2013}
H.~Y\"ucel, M.~Heinkenschloss, and B.~Karas\"{o}zen, ``Distributed optimal
  control of diffusion-convection-reaction equations using discontinuous
  {G}alerkin methods,'' in {\em Numerical Mathematics and Advanced Applications
  2011}, (Berlin), pp.~389--397, Springer, 2013.

\bibitem{ZZhou_XYu_NYan_2014a}
Z.~Zhou, X.~Yu, and N.~Yan, ``The local discontinuous {G}alerkin approximation
  of convection-dominated diffusion optimal control problems with control
  constraints,'' {\em Numer. Methods. Partial Differential Equations}, vol.~30,
  no.~1, pp.~339--360, 2014.

\bibitem{RAAdams_1975}
R.~A. Adams, {\em Sobolev Spaces}.
\newblock Orlando, San Diego, New-York: Academic Press, 1975.

\bibitem{ECasas_MMateos_JPRaymond_2009}
E.~Casas, M.~Mateos, and J.-P. Raymond, ``Penalization of {D}irichlet optimal
  control problems,'' {\em ESAIM Control Optim. Calc. Var.}, vol.~15, no.~4,
  pp.~782--809, 2009.

\bibitem{KKunisch_BVexler_2007c}
K.~Kunisch and B.~Vexler, ``Constrained {D}irichlet boundary control in ${L}^2$
  for a class of evolution equations,'' {\em SIAM J. Control Optim.}, vol.~46,
  no.~5, pp.~1726--1753, 2007.


\bibitem{JLLions_1971}
J.-L. Lions, {\em Optimal Control of Systems Governed by Partial Differential
  Equations}.
\newblock Berlin: Springer, 1971.

\bibitem{CCarstensen_1999}
C.~Carstensen, ``Quasi-interpolation and a posteriori error analysis in finite
  element methods,'' {\em ESAIM:M2AN}, vol.~36, no.~6, pp.~1197--1202, 1999.

\bibitem{MBergounioux_KIto_KKunisch_1999a}
M.~Bergounioux, K.~Ito, and K.~Kunisch, ``Primal-dual strategy for constrained
  optimal control problems,'' {\em SIAM J. Control Optim.}, vol.~37, no.~4,
  pp.~1176--1194, 1999.

\end{thebibliography}
